\documentclass[a4paper]{amsart}

\usepackage{amsmath,amsfonts,amssymb,amsthm} %AMS packages
\usepackage{mathrsfs} %Script fonts
\usepackage[shortlabels]{enumitem} %Adds extra options to enumerate environment
\usepackage{caption} %Gives better caption formatting and more caption formatting options
\usepackage[T1]{fontenc} %Extra font characters

\usepackage{mathtools} %Fixes some problems with amsmath and adds some commands
\mathtoolsset{centercolon} %Corrects colons before equals signs

\usepackage{tikz} %TikZ for drawing things
\usetikzlibrary{cd} %Commutative diagrams

\usepackage{todonotes}
\usepackage{diagbox}
\usepackage[thinlines]{easytable}

\usepackage{hyperref} %Hyperlinks

\theoremstyle{definition}
\newtheorem{definition}{Definition}[section]

\theoremstyle{plain}
\newtheorem{lemma}[definition]{Lemma}
\newtheorem{theorem}[definition]{Theorem}
\newtheorem{proposition}[definition]{Proposition}

\newtheorem{conjecture}[definition]{Conjecture}

\theoremstyle{remark}
\newtheorem{remark}[definition]{Remark}
\newtheorem{example}[definition]{Example}

\newcommand{\bA}{\mathbb{A}}

\newcommand{\bC}{\mathbb{C}}

\newcommand{\bH}{\mathbb{H}}

\newcommand{\bN}{\mathbb{N}}

\newcommand{\bP}{\mathbb{P}}
\newcommand{\bQ}{\mathbb{Q}}

\newcommand{\bZ}{\mathbb{Z}}

\newcommand{\calF}{\mathcal{F}}

\newcommand{\calH}{\mathcal{H}}

\newcommand{\calL}{\mathcal{L}}
\newcommand{\calM}{\mathcal{M}}

\newcommand{\calO}{\mathcal{O}}

\newcommand{\calX}{\mathcal{X}}

\DeclareMathOperator{\coker}{coker}

\DeclareMathOperator{\Gr}{Gr}
\DeclareMathOperator{\Hom}{Hom}
\DeclareMathOperator{\im}{im}
\DeclareMathOperator{\coim}{coim}

\DeclareMathOperator{\Nef}{Nef}

\DeclareMathOperator{\rank}{rank}

\numberwithin{figure}{section}
\numberwithin{table}{section}

\begin{document}

\title[The Mirror Clemens-Schmid Sequence]{The Mirror Clemens-Schmid Sequence}

\author[C. F. Doran]{Charles F. Doran}
\address{Department of Mathematical and Statistical Sciences, 632 CAB, University of Alberta,  Edmonton, AB, T6G 2G1, Canada. \newline \indent Bard College, Annandale-on-Hudson, NY 12571, USA. \newline \indent Center of Mathematical Sciences and Applications, Harvard University, 20 Garden Street, Cambridge, MA 02138, USA.}
\email{charles.doran@ualberta.ca}
\thanks{Charles F. Doran (Alberta/Bard College/CMSA Harvard) was supported by the Natural Sciences and Engineering Research Council of Canada (NSERC)}

\author[A. Thompson]{Alan Thompson}
\address{Department of Mathematical Sciences, Loughborough University, Loughborough, Leicestershire, LE11 3TU, United Kingdom.}
\email{A.M.Thompson@lboro.ac.uk}
\thanks{Alan Thompson (Loughborough) was supported by Engineering and Physical Sciences Research Council (EPSRC) New Investigator Award EP/V005545/1.}

\begin{abstract}We introduce a four-term long exact sequence that relates the cohomology of a smooth variety admitting a projective morphism onto a projective base to the cohomology of the open set obtained by removing the preimage of a general linear section. We show that this sequence respects the perverse Leray filtration and induces exact sequences of mixed Hodge structures on its graded pieces. We conjecture that this exact sequence should be thought of as mirror to the Clemens-Schmid sequence, which describes the cohomology of degenerations. We exhibit this mirror relationship explicitly for all Type II and many Type III degenerations of K3 surfaces. In three dimensions, we show that for Tyurin degenerations of Calabi-Yau threefolds our conjecture is a consequence of existing mirror conjectures, and we explicitly verify our conjecture for a more complicated degeneration of the quintic threefold.
\end{abstract}


\date{}
\maketitle

\section{Introduction}

\subsection{Motivation} The motivation behind this work comes from mirror symmetry. In \cite{mstdfcym} the authors, along with A. Harder, postulated a mirror symmetric correspondence between a certain type of degeneration of a Calabi-Yau variety, called a \emph{Tyurin degeneration}, and codimension one fibrations on the mirror Calabi-Yau variety. This correspondence is expected to identify the action of the logarithm of monodromy around the degenerate fibre on the degeneration side with the cup product with a fibre on the fibration side.

An important tool for studying the structure of degenerations is the \emph{Clemens-Schmid sequence}: a four-term long exact sequence of mixed Hodge structures that relates the geometry of the central fibre of the degeneration to the geometry of a nearby smooth fibre and encodes the action of the logarithm of monodromy around the degenerate fibre. This gives rise to the following natural question: \emph{Is there a mirror to the Clemens-Schmid sequence that encodes the structure of a fibration and the action of the cup-product map?}

Part of the structure of this mirror sequence is clear: the cohomology of a nearby fibre, equipped with the action of the logarithm of monodromy, should be mirror to the cohomology of the total space of the fibration, equipped with the action of cup product with a fibre. Moreover, classical mirror symmetry tells us how to mirror the Hodge filtration: this is the usual exchange of Hodge numbers from which mirror symmetry takes its name. However, the mirror to the monodromy weight filtration is less clear. In \cite{pwp}, Katzarkov, Przyjalkowski, and Harder postulate that, in the context of Fano-LG mirror symmetry, the mirror to a monodromy weight filtration should be a perverse Leray filtration. The cohomology of any fibred variety can be equipped with a perverse Leray filtration in a natural way, so this seems like a good candidate in our setting too. One might therefore expect the mirror to the Clemens-Schmid sequence to be a four-term sequence on the cohomology of a fibred variety which respects the Hodge and perverse Leray filtrations and includes the cup product with a fibre.

%As further evidence for our approach, in the introduction of \cite{htam} de Cataldo and Migliorini noted that the perverse Leray filtration and the action of cup product display a ``striking similarity with the structures discovered on the cohomology of the limit fiber... in the case of degenerating families along a normal crossing divisor, where the logarithms of the monodromies, instead of the cup products with Chern classes of line bundles, are the endomorphisms giving the filtrations''. We claim that this similarity is, in fact, a manifestation of mirror symmetry.

\subsection{Description of the results} The following is the main result of this paper.

\begin{theorem} \label{thm:mainthm} Let $\pi \colon Y \to B$ be a projective, surjective morphism from a smooth complex projective variety $Y$ to a complex projective variety $B$. Fix an embedding $B \subset \bP^N$ and let $Z \subset Y$ be the preimage under $\pi$ of a general linear subspace. Let $U := Y - Z$ denote the complement of $Z$ and let $j\colon U \hookrightarrow Y$ denote the inclusion. Then the four-term sequence
\[\begin{tikzcd}[cramped]\to H^{k}_c(U) \ar[r,"{j_!}"] & H^{k}(Y) \ar[r,"{[Z]\cup -}"] & H^{k+2}(Y)(1) \ar[r,"{j^*}"]  & H^{k+2}(U)(1) \ar[r]& H^{k+2}_c(U) \to
\end{tikzcd}\]
is an exact sequence of mixed Hodge structures on the graded pieces of the perverse Leray filtration (here $(1)$ denotes the usual Tate twist).
\end{theorem}

Motivated by the discussion above, we call this sequence the \emph{mirror Clemens-Schmid sequence}. Note, however, that Theorem \ref{thm:mainthm} holds in much greater generality than the motivating situation presented above might suggest: $\pi \colon Y \to B$ is any projective, surjective morphism, not just a codimension $1$ fibration. On the other hand, the usual Clemens-Schmid sequence also holds for arbitrary semistable degenerations, not just Tyurin degenerations. Comparing the two sequences leads us to a general conjecture relating semistable degenerations of a Calabi-Yau to surjective morphisms on its mirror. To state this conjecture formally, we first give some definitions.

\begin{definition} \label{def:degeneration} A \emph{degeneration} is a proper, flat, surjective morphism $\calX \to \Delta$ from a K\"{a}hler manifold to the complex unit disc, which has connected fibres and is smooth over the punctured disc $\Delta^*$.  Denote a general fibre by $X$ and denote the fibre over $0 \in \Delta$ by $X_0$. A degeneration $\calX \to \Delta$ is called \emph{semistable} if $X_0 \subset \calX$ is a simple normal crossings divisor.
\end{definition}

Given a semistable degeneration, the logarithm of monodromy induces a morphism on the limiting mixed Hodge structure
\[\nu \colon H^{*}_{\lim}(X, \bQ) \longrightarrow H^{*}_{\lim}(X,\bQ)(-1)\]
and this action is nilpotent with $\nu^{\dim(X) + 1} = 0$ (see \cite[\S 11.3]{mhs}).

\begin{definition} \label{def:length}
The \emph{length} of a semistable degeneration $\calX \to \Delta$ is the smallest $\lambda \in \bN$ such that $\nu^{\lambda} = 0$.
\end{definition}

 We will primarily be interested in degenerations of Calabi-Yau varieties, i.e. degenerations where the general fibre $X$ is a nonsingular complex variety with $\omega_X \cong \calO_X$ and $h^i(X,\calO_X) = 0$ for all $0 < i < \dim_{\bC}(X)$.

\begin{definition} \label{def:Ktrivial} A degeneration $\calX \to \Delta$ of Calabi-Yau varieties is called \emph{$K$-trivial} if $\omega_{\calX} \cong \calO_{\calX}$.
\end{definition}

Note that $K$-trivial semistable degenerations of K3 surfaces of lengths $1$, $2$, and $3$ correspond precisely to degenerations of Types I, II, and III, respectively (see \cite[Section 4(d)]{csesa}). A special class of $K$-trivial semistable degenerations are the \emph{Tyurin degenerations}.

\begin{definition} \label{def:tyurin} A smooth variety $V$ is called \emph{quasi-Fano} if its anticanonical linear system contains a smooth Calabi-Yau member and $H^k(V,\calO_V)=0$ for all $k > 0$. Given this, a \emph{Tyurin degeneration} is a $K$-trivial semistable degeneration of Calabi-Yau varieties $\calX \to \Delta$ whose central fibre $X_0$ is a union of two quasi-Fano varieties $V_1 \cup_S V_2$ meeting normally along a smooth member of the anticanonical linear system $S \in |-K_{V_i}|$ in each. 
\end{definition}

Note that Tyurin degenerations have length $2$. With this in place we can state our conjecture precisely.

\begin{conjecture} \label{con:mirrorfiltrations} Suppose $\calX \to \Delta$ is a $K$-trivial semistable degeneration of Calabi-Yau varieties of length $\lambda$. Let $Y$ be a Calabi-Yau variety mirror to a general fibre $X$. Then there is a projective, surjective morphism $\pi\colon Y \to B$ with $\dim_{\bC}(B) = \lambda-1$ and we have an equality between the dimensions of the graded pieces
\[\dim(\Gr_F^p\Gr^W_q\Gr^P_lH^k(Y)) = \dim(\Gr_F^{n-p}\Gr^W_{n+l-2p}\Gr^P_{n+q-2p}H^{n+k-2p}_{\lim}(X)),\]
where $n = \dim_{\bC}(X)$ and $F^{\bullet}$, $W_{\bullet}$, and $P_{\bullet}$ denote the Hodge, weight, and perverse Leray filtrations respectively.

Moreover, for an appropriate choice of embedding of $B$ into projective space and with $U \subset Y$ the complement of the preimage of a general linear subspace, we have 
\[\dim(\Gr_F^p\Gr^W_q\Gr^P_lH^k_c(U)) = \dim(\Gr_F^{n-p}\Gr^W_{n+l-2p}\Gr^P_{n+q-2p+1}H^{n+k-2p}(\calX)).\]
\end{conjecture}

As evidence for this conjecture we will prove two results. The first focusses on the K3 surface case.

\begin{theorem} \label{thm:K3evidence} \begin{enumerate}
\item If $\calX \to \Delta$ is a Type II degeneration of lattice polarized K3 surfaces with general fibre $X$, then there is a mirror $Y$ to $X$ (in the sense of \cite{mslpk3s}) admitting a genus $1$ fibration $\pi\colon Y \to \bP^1$ for which Conjecture \ref{con:mirrorfiltrations} holds.
\item If $\calX \to \Delta$ is a Type III degeneration of K3 surfaces polarized by the lattice $H \oplus E_8 \oplus E_8 \oplus \langle -2k \rangle$ and $Y$ is a generic K3 surface of degree $2k$, then there is a finite morphism $\pi\colon Y \to B$ for which Conjecture \ref{con:mirrorfiltrations} holds.
\item Let $\calX \to \Delta$ be a Type II or III degeneration of K3 surfaces with general fibre $X$. Suppose that the mirror $Y$ to $X$ admits a morphism $\pi\colon Y \to B \subset \bP^N$ for which Conjecture \ref{con:mirrorfiltrations} holds. Then Conjecture \ref{con:mirrorfiltrations} also holds between the $\mu$-fold base change of $\calX$ and the degree $\mu$ Veronese re-embedding  $Y \to B \subset \bP^N \hookrightarrow \bP^M$.
\end{enumerate}
\end{theorem}

Our final result focusses on the case of Calabi-Yau threefolds. In this setting we have the following conjecture from \cite[Section 2.3]{mstdfcym}.

\begin{conjecture}\label{con:dht} Let $\calX \to \Delta$ be a Tyurin degeneration of Calabi-Yau threefolds with general fibre $X$ and central fibre $X_0 = V_1 \cup_S V_2$. Then there is a mirror $Y$ to $X$ which admits a K3-fibration $\pi\colon Y \to \bP^1$. Moreover, the general K3 fibre $Z$ of $\pi$, with lattice polarization given by the monodromy invariant classes in $H^2(Z,\bZ)$, is mirror (in the sense of \cite{mslpk3s}) to $S$, with lattice polarization induced by restriction of classes from $H^2(V_1)$ and $H^2(V_2)$.
\end{conjecture}

Our final result shows that this conjecture implies Conjecture \ref{con:mirrorfiltrations} in the Tyurin degeneration setting. We also verify Conjecture \ref{con:mirrorfiltrations} for an explicit degeneration of length $3$.

\begin{theorem} \label{thm:CY3evidence} \begin{enumerate}
\item If $\calX \to \Delta$ is a Tyurin degeneration of Calabi-Yau threefolds and $\pi\colon Y \to \bP^1$ is a K3-fibration on its mirror which satisfies Conjecture \ref{con:dht}, then Conjecture \ref{con:mirrorfiltrations} holds for $\calX \to \Delta$  and $\pi\colon Y \to \bP^1$.
\item There is a degeneration of length $3$ of the quintic threefold and an elliptic fibration on its mirror for which Conjecture \ref{con:mirrorfiltrations} holds.
\end{enumerate}
\end{theorem} 

\subsection{Relationship to other work} 

The first to study the question of finding a mirror to the monodromy weight filtration was probably Gross \cite{slfit}. Following the SYZ approach to mirror symmetry he conjectured that, in dimensions $2$ and $3$, the monodromy weight filtration around a boundary divisor in the complex moduli space of a Calabi-Yau manifold $X$ should coincide with the Leray filtration associated to a special Lagrangian torus fibration on $X$ and, moreover, that this should be mirror to the Leray filtration associated to the dual special Lagrangian torus fibration on the SYZ mirror $\check{X}$. However, he also notes \cite[Remark 4.7]{slfit} that there are difficulties in formulating this in higher dimensions, as it is unclear whether the fibrations involved have the Hard Lefschetz property. Gross and Siebert overcame these issues in \cite{mslddii}, using an approach that bypasses the Leray filtration and instead works directly with the cup product, to describe a mirror (in the sense of the Gross-Siebert program) to the monodromy weight filtration in any dimension. 

Harder, Katzarkov, and Przyjalkowski \cite{pwp} suggest a different approach. They conjecture that the weight filtration on the cohomology of a log Calabi-Yau variety $U$ is instead mirror to a \emph{perverse} Leray filtration on the cohomology of the (homological) mirror log Calabi-Yau variety $\check{U}$. This approach solves the Hard Lefschetz problem encountered by Gross: unlike for the Leray filtration, the Hard Lefschetz Theorem holds for the perverse Leray filtration associated to any projective morphism from a nonsingular variety \cite[Theorem 2.1.4]{htam}. 

This is related to the SYZ picture studied by Gross as follows. In the case where $U$ is the complement of a smooth anticanonical divisor in a nonsingular projective variety, one expects $U$ to admit a special Lagrangian torus fibration. SYZ mirror symmetry then postulates that the mirror $\check{U}$ should admit a dual special Lagrangian torus fibration. The conjecture of \cite{pwp} can be thought of as relating the weight filtration on the cohomology of $U$ to the perverse Leray filtration associated to this dual special Lagrangian torus fibration on $\check{U}$.

This conjecture is called the ``mirror $P=W$'' conjecture in \cite{pwp}. We now discuss how it is related to the ``classical'' $P=W$ conjecture, due to de Cataldo, Hausel, and Migliorini \cite{thshtcva1}. Briefly, if $C$ is a smooth complex projective curve, $\calM^H_{n,d}(C)$ is the moduli space of rank $n$ degree $d$ semistable Higgs bundles on $C$, and $\calM^B_{n,d}(C)$ is the Betti moduli space of $d$-twisted $n$-dimensional representations of $\pi_1(C)$, then $\calM^H_{n,d}(C)$ and $\calM^B_{n,d}(C)$ are diffeomorphic quasiprojective varieties that are not deformation equivalent. The cohomology of $\calM^B_{n,d}(C)$ admits a mixed Hodge structure with weight filtration $W_{\bullet}$, whilst $\calM^H_{n,d}(C)$ admits a Hitchin map $\calM^H_{n,d}(C) \to \bC^k$, where $\dim_{\bC}\calM^H_{n,d}(C) = 2k$, which induces a perverse Leray filtration $P_{\bullet}$ on its cohomology. The classical $P=W$ conjecture then states that
\[W_{2i}H^n(\calM^B_{n,d}(C)) = W_{2i+1}H^n(\calM^B_{n,d}(C)) = P_iH^n(\calM^H_{n,d}(C)).\]
This conjecture has recently been proved independently in \cite{pwh2} and \cite{pwcgln}.

To see how the classical and mirror $P=W$ conjectures are related, note that the moduli space $\calM^B_{n,d}(C)$ is expected to be log Calabi-Yau \cite{dbcsl2cvps} and admits a special Lagrangian torus fibration induced by the Hitchin fibration. The $P=W$ conjecture then relates the weight filtration on $H^n(\calM^B_{n,d}(C))$ to the perverse Leray filtration on $H^n(\calM^H_{n,d}(C))$ associated to this special Lagrangian torus fibration. The parallel between the two conjectures is clear; for more detail see \cite{pwp}.

Our Conjecture \ref{con:mirrorfiltrations} can be thought of as a version of the mirror $P=W$ conjecture for compact Calabi-Yau varieties, rather than the log Calabi-Yau varieties studied in \cite{pwp}. However, with the exception of elliptic fibrations on K3 surfaces, the fibrations we study here are generally \emph{not} the special Lagrangian torus fibrations which appear in the work of Gross and in the classical $P=W$ conjecture.

\subsection{Structure of the paper} Section \ref{sec:perverse} consists largely of background material on the perverse Leray filtration on cohomology. After setting up notation, we prove that this filtration is respected by two localization exact sequences (Proposition \ref{prop:perverselocalization}). We then briefly review some results of de Cataldo and Migliorini \cite{pflht} (Theorem \ref{thm:geometricleray}), which give a geometric interpretation of the perverse Leray filtration, before using these to prove a series of useful vanishing lemmas for the graded pieces. We conclude the section with a result from \cite{htam} (Proposition \ref{prop:cupperverse}), which shows that the perverse Leray filtration on the cohomology of $Y$ coincides with the one induced by cup-product with the preimage of a general linear section.

Section \ref{sec:mirrorCS} contains the proof of Theorem \ref{thm:mainthm}. Here we construct the mirror Clemens-Schmid sequence and prove that it is an exact sequence of mixed Hodge structures on the graded pieces of the perverse Leray filtration.

Section \ref{sec:mirror} presents the relationship with mirror symmetry. We begin with a review of the Clemens-Schmid sequence and the results of \cite{htdicdt}, which show that it is also an exact sequence of mixed Hodge structures on the graded pieces of the perverse Leray filtration. Comparing this to the mirror Clemens-Schmid sequence, noting that mirror symmetry is expected to exchange Hodge numbers and exchange the weight and perverse Leray filtrations, gives rise to Conjecture \ref{con:mirrorfiltrations}.

In Section \ref{sec:K3} we prove Theorem \ref{thm:K3evidence} and in Section \ref{sec:CY3} we prove Theorem \ref{thm:CY3evidence}. In each case the proof proceeds by explicit computation of the filtrations involved.

\subsection{Acknowledgments} The authors had many helpful conversations during the development of this paper; we would particularly like to thank Andrew Harder, Sukjoo Lee, James Lewis, Matt Kerr, and Helge Ruddat for taking the time to speak with us. 

\subsection{Data availability} Data sharing not applicable to this article as no datasets were generated or analysed during the current study.

\subsection{Notation and assumptions}\label{sec:notation}

The following notation will be used throughout this paper.
\begin{itemize}
\item $Y$ is a smooth complex projective variety of complex dimension $n$.
\item $B$ is a complex projective variety of complex dimension $m$. We fix an embedding $B \subset \bP^N$. Note that we do not require $B$ to be smooth.
\item $\pi \colon Y \to B$ is a projective, surjective morphism. Note that, whilst we often consider the case where $\pi$ is a fibration, in general we do not require $\pi$ to be flat.
\item $C \subset B$ is a general linear subspace, given by intersecting $B$ with a general hyperplane in $\bP^N$.
\item $Z := \pi^{-1}(C) \subset Y$. As the preimage of a general linear subspace, $Z$ is smooth by Bertini's Theorem (although not necessarily connected!).
\item $U := Y - Z$ is the preimage of the affine variety $B - C$ under $\pi$.
\item The restrictions of $\pi$ to $Z$ and $U$ are denoted by $\pi^Z$ and $\pi^U$, respectively.
\end{itemize}

%$E_2$ degeneration for the perverse Leray sequence still holds when $B$ is not smooth, as long as $Y$ is. This holds as long as $\underline{\bQ}_Y[n]$ is semisimple, which is true if $Y$ is smooth, and $\pi$ is projective; see \cite[Section 3]{htdicdt}.

\section{The perverse Leray filtration}\label{sec:perverse}

\subsection{The perverse Leray filtration and localization}

We begin with a discussion of the perverse Leray filtration; for full details of this construction we refer the reader to \cite[Chapters 13 and 14]{mhs}.

Consider a morphism $\pi\colon Y \to B$ as in Section \ref{sec:notation}. The \emph{perverse Leray spectral sequence} associated to $\pi\colon Y \to B$ is the spectral sequence with $E_2$-term
\[E_2^{p,q} := \bH^p(B,\prescript{\mathfrak{p}}{}{\calH}^q(R\pi_*\underline{\bQ}_Y)) \Rightarrow H^{p+q}(Y,\bQ),\]
where $\prescript{\mathfrak{p}}{}{\calH}^q$ denotes the perverse cohomology functor. By \cite[Corollary 14.13]{mhs}, this is in fact a spectral sequence of mixed Hodge structures, which is compatible with the standard mixed Hodge structure on $H^{p+q}(Y,\bQ)$. 

Unlike the standard Leray spectral sequence, the perverse Leray spectral sequence has the wonderful property that it degenerates at the $E_2$-term for any projective morphism \cite[Corollary 14.41]{mhs} (see also \cite[Section 3]{htdicdt} for the same statement under much weaker assumptions). This gives a decomposition
\[H^k(Y,\bQ) = \bigoplus_{p+q=k} E_2^{p,q}.\]
Using this, we give the following standard definition.

\begin{definition} \label{def:perverseleray} The \emph{perverse Leray filtration} $P_{\bullet}$ on $H^*(Y,\bQ)$ is the increasing filtration with graded pieces
\[\Gr_l^PH^k(Y,\bQ) := E_2^{k-l,l} = \bH^{k-l}(B,\prescript{\mathfrak{p}}{}{\calH}^l(R\pi_*\underline{\bQ}_Y)).\]
\end{definition}

\begin{remark} The statements above also hold for the morphisms $\pi^Z$ and $\pi^U$, allowing us to define perverse Leray filtrations on $H^*(Z,\bQ)$ and $H^*(U,\bQ)$ in the obvious way. Moreover, an analogous construction for the compactly supported cohomology $H^*_c(U,\bQ)$ can be obtained by replacing $R\pi_*$ in the definitions above with the compactly supported direct image $R\pi_!^U$.
\end{remark}

\begin{remark} Since middle perversity is preserved by duality when working over a field, it follows that the perverse Leray filtrations for $H^*(U,\bQ)$ and $H^*_c(U,\bQ)$ are exchanged by Poincar\'{e}-Verdier duality, whilst those for $H^*(Y,\bQ)$ and $H^*(Z,\bQ)$ are preserved.
\end{remark}

Now we state the main result of this section.

\begin{proposition}\label{prop:perverselocalization} The localization long exact sequences
\[\cdots\longrightarrow H^k(Y,\bQ) \longrightarrow H^k(U,\bQ) \longrightarrow H^{k-1}(Z,\bQ)[-1](-1) \longrightarrow H^{k+1}(Y,\bQ)\longrightarrow\cdots \]
and
\[ \cdots \longrightarrow H^k_c(U,\bQ) \longrightarrow H^k(Y,\bQ) \longrightarrow H^k(Z,\bQ)[-1] \longrightarrow H^{k+1}_c(U,\bQ) \longrightarrow \cdots,\]
induce exact sequences of mixed Hodge structures on the graded pieces of the perverse Leray filtration. Here $[-1]$ denotes a shift in the perverse Leray filtration and $(-1)$ denotes the usual Tate twist on the mixed Hodge structure.
\end{proposition}
\begin{proof}
We first note that it suffices to prove that the sequences are exact on the graded pieces of $P_{\bullet}$: the statement about mixed Hodge structures follows immediately from the fact  that the mixed Hodge structures on the graded pieces are induced from the standard ones on cohomology \cite[Corollary 14.13]{mhs}, combined with the well-known fact that the localization sequences are exact sequences of mixed Hodge structures.

To prove exactness on graded pieces of $P_{\bullet}$ for the first sequence, we begin by noting that $\prescript{\mathfrak{p}}{}\calH^l(R\pi_*\underline{\bQ}_Y)$ is a perverse sheaf on $B$. Thus, as $B$ is projective and $C$ is a general linear section, by \cite[Lemma 5.3.1]{pflht} we have a natural localization exact sequence
\[0 \longrightarrow \prescript{\mathfrak{p}}{}\calH^l(R\pi_*\underline{\bQ}_Y) \longrightarrow Rj_*j^*\prescript{\mathfrak{p}}{}\calH^l(R\pi_*\underline{\bQ}_Y) \longrightarrow Ri_!i^!\prescript{\mathfrak{p}}{}\calH^l(R\pi_*\underline{\bQ}_Y)[1] \longrightarrow 0,\]
where $i\colon C \hookrightarrow B$ (resp. $j\colon B- C \hookrightarrow B$) is the natural open (resp. closed) embedding.

The functors $j^*$, $Rj_*$, $Ri_!$, $i^!$ satisfy the following exactness properties with respect to the middle perversity $t$-structure:
\begin{itemize}
\item $j^* = j^!$ is $t$-exact as $j$ is an open immersion;
\item as $j$ is an affine embedding, $Rj_*$ is $t$-exact \cite[Section 3.7]{htam};
\item $Ri_! = Ri_*$ is $t$-exact as $i$ is a closed immersion;
\item $i^![1]=i^*[-1]$ is $t$-exact as $i$ is a normally nonsingular inclusion of complex varieties of complex codimension $1$ \cite[Lemma 3.5.4]{htam}.
\end{itemize}
In particular, these functors commute with the perverse cohomology functor, so we have an exact sequence
\[0 \longrightarrow \prescript{\mathfrak{p}}{}\calH^l(R\pi_*\underline{\bQ}_Y) \longrightarrow Rj_*\prescript{\mathfrak{p}}{}\calH^l(j^*R\pi_*\underline{\bQ}_Y) \longrightarrow Ri_!\prescript{\mathfrak{p}}{}\calH^l(i^!R\pi_*\underline{\bQ}_Y[1]) \longrightarrow 0.\]
By a simple base-change argument, using the properties of the four functors listed above, we may rewrite this sequence as
\[0 \longrightarrow \prescript{\mathfrak{p}}{}\calH^l(R\pi_*\underline{\bQ}_Y) \longrightarrow Rj_*\prescript{\mathfrak{p}}{}\calH^l(R\pi^U_*\underline{\bQ}_U) \longrightarrow Ri_*\prescript{\mathfrak{p}}{}\calH^{l}(R\pi_*^Z\underline{\bQ}_Z[-1]) \longrightarrow 0.\]
Finally, taking hypercohomology of this sequence on $B$ gives the required sequence.

The proof for exactness of the second sequence follows by applying an analogous argument to the second localization sequence from \cite[Lemma 5.3.1]{pflht}, or by applying Poincar\'{e}-Verdier duality.\end{proof}

\subsection{A geometric description of the perverse Leray filtration}\label{sec:geometricfiltrations}

The aim of this section is to describe how the results of \cite{pflht} may be used to give a geometric description of the perverse Leray filtration.

We begin by setting up some more notation. There is a smooth projective variety $F(N,m)$ parametrizing $m$-flags $\calF = \{\Lambda_{-1} \subset \cdots \subset \Lambda_{-m}\}$ on $\bP^N$, where $\Lambda_{-p} \subset \bP^N$ is a codimension $p$ linear subspace. A linear $m$-flag $\calF$ on $\bP^N$ is \emph{general} if it belongs to a suitable Zariski open subset of $F(N,m)$ and, similarly, a pair of flags is \emph{general} if the same is true of the pair with respect to $F(N,m) \times F(N,m)$; for precise details of these open sets, we refer the reader to \cite[Section 5.2]{pflht}.

Fix a general pair of linear $m$-flags on $\bP^N$. Intersecting with $B$, we obtain a pair of $m$-flags $B_{\bullet}$ and $C_{\bullet}$ on $B$; i.e. a pair of increasing sequences of subvarieties of $B$:
\begin{align*}
B &= B_0 \supset B_{-1} \supset \cdots \supset B_{-m} \supset B_{-m-1} = \emptyset,\\
B &= C_0 \supset C_{-1} \supset \cdots \supset C_{-m} \supset C_{-m-1} = \emptyset.
\end{align*}
Setting $Y_p = \pi^{-1}B_p$ and $Z_p = \pi^{-1}C_p$, we obtain a pair of increasing sequences of subvarieties of $Y$:
\begin{align*}
Y &= Y_0 \supset Y_{-1} \supset \cdots \supset Y_{-m} \supset Y_{-m-1} = \emptyset,\\
Y &= Z_0 \supset Z_{-1} \supset \cdots \supset Z_{-m} \supset Z_{-m-1} = \emptyset.
\end{align*}
We assume, without loss of generality, that $C = C_{-1}$ and $Z = Z_{-1}$; consequently we have $U = Y - Z_{-1}$.

Let $f \colon X \to Y$ be a locally closed embedding. Then we have the restriction functor, denoted $(-)|_X$%= Rf_!f^*$ on $D_Y$
, which is exact. 
If $f$ is closed, then we also have the right derived functor of sections with support in $X$, denoted $R\Gamma_X(-)$.% = Rf_*f^!$.

%Following \cite[Section 3.6]{pflht}, we now define the \emph{flag} and \emph{$\delta$-filtrations} on cohomology. We begin with the flag filtrations on the cohomology of $U$.
%
%\begin{definition}
%The \emph{flag filtration} $G^{\bullet}$ on the cohomology of $U$ is the decreasing filtration defined on regular cohomology by
%\[G^{p}H^k(U,\bQ) := \ker\left\{\bH^k(U,\underline{\bQ}) \longrightarrow \bH^k(U,\underline{\bQ}|_{U\cap Y_{p-1}})\right\}\]
%and on compactly supported cohomology by
%\[G^pH_c^k(U,\bQ) := \im\left\{\bH^k_c(U,R\Gamma_{U\cap Y_{-p}}\underline{\bQ}) \longrightarrow \bH_c^k(U,\underline{\bQ})\right\}.\]
%\end{definition}
%
%Next we define the $\delta$-filtration on the cohomology of $Y$.
%
%\begin{definition}
%The \emph{$\delta$-filtration} $\delta^{\bullet}$ is the decreasing filtration defined on the cohomology of $Y$ by
%\[\delta^pH^k(Y,\bQ) := \im\left\{\bigoplus_{i+j=p} \bH^k(Y,R\Gamma_{Y_{-i}}(\underline{\bQ}|_{Y-Z_{j-1}})) \longrightarrow \bH^k(Y,\underline{\bQ})\right\}.\]
%%and on the cohomology of $Z$ by
%%\[\delta^pH^k(Z,\bQ) := \im\left\{\bigoplus_{i+j=p} \bH^k(Z,R\Gamma_{Z \cap Y_{-i}}(\bQ_{Z-Z_{j-2}})) \longrightarrow \bH^k(Z,\bQ)\right\}\]
%\end{definition}

The following theorem, due to de Cataldo and Migliorini, gives a geometric interpretation of the perverse Leray filtration.

\begin{theorem}\label{thm:geometricleray}\cite[Theorems 4.1.3 and 4.2.1]{pflht}
The perverse Leray filtrations $P_{\bullet}$ on the cohomology of $U$ and $Y$ 
%coincide with the shifted flag and $\delta$-filtrations as follows:
may be computed from the flags $Y_{\bullet}$ and $Z_{\bullet}$ as follows:
\begin{align*}
P_lH^k(U,\bQ) &= %G^{k-l}H^k(U,\bQ) = 
\ker\left\{\bH^k(U,\underline{\bQ}) \longrightarrow \bH^k(U, \underline{\bQ}|_{U\cap Y_{k-l-1}})\right\},\\
P_lH^k_c(U,\bQ) &= %G^{k-l}H_c^k(U,\underline{\bQ}) = 
\im\left\{\bH^k_c(U,R\Gamma_{U \cap Y_{l-k}}\underline{\bQ}) \longrightarrow \bH_c^k(U,\underline{\bQ})\right\},\\
 P_lH^k(Y,\bQ) &= %\delta^{k-l}H^k(Y,\bQ) \\ &= 
 \im\left\{\bigoplus_{i+j=k-l} \bH^k(Y,R\Gamma_{Y_{-i}}(\underline{\bQ}|_{Y-Z_{j-1}})) \to \bH^k(Y,\underline{\bQ})\right\},\\
 %P_lH^k(Z,\bQ) &= \delta^{k-l}H^k(Z,\bQ) \\
% &= \im\left\{\bigoplus_{i+j=k-l} \bH^k(Z,R\Gamma_{Z \cap Y_{-i}}(\bQ_{Z-Z_{j-2}})) \longrightarrow \bH^k(Z,\bQ)\right\}.
 \end{align*}
\end{theorem}

\begin{remark} Note that the formulas in Theorem \ref{thm:geometricleray} differ from those in \cite{pflht} by a sign; this is because we follow the convention that the perverse Leray filtration is increasing, but \cite{pflht} define it to be decreasing.
\end{remark}

\begin{remark} \label{rem:Zfiltration} One may modify Theorem \ref{thm:geometricleray} in the obvious way to compute the perverse Leray filtrations on the cohomology of the subvarieties $Z_{-r}$. In this case the filtration is induced by the pair of increasing sequences of subvarieties
\begin{align*}
Z_{-r} &= (Y_0 \cap Z_{-r}) \supset (Y_{-1} \cap Z_{-r}) \supset \cdots \supset (Y_{r-m-1} \cap Z_{-r}) = \emptyset,\\
Z_{-r} &\supset \cdots \supset Z_{-m} \supset Z_{-m-1} = \emptyset
\end{align*}
given by intersection of $Y_{\bullet}$ and $Z_{\bullet}$ with $Z_{-r}$. Note that there is a shift in indices in the $Z_{\bullet}$-filtration. In the case $r = 1$ this provides a geometric explanation for the shift in the perverse Leray filtration on $H^k(Z,\bQ)$ in the sequences of Proposition \ref{prop:perverselocalization}.
\end{remark}

\subsection{Isomorphisms and vanishing}

In this section we prove some useful lemmas about the graded pieces of the perverse Leray filtration.

\begin{lemma}\label{lem:Uvanishing}
Recall that $m := \dim_{\bC}(B)$. We have the following vanishing for the graded pieces of the perverse Leray filtration on $U$:
\begin{itemize}
\item $\Gr_l^PH^k(U,\bQ) = 0$ whenever $l < k$ or $l > k + m$,
\item $\Gr_l^PH^k_c(U,\bQ) = 0$ whenever $l < k - m$ or $l > k$.
\end{itemize}
\end{lemma}
\begin{proof}
We begin by proving the statements for $U$. Note that if $l < k$, we have $k-l-1 \geq 0$ and so $Y_{k-l-1} = Y$. Thus $\bH^k(U, \underline{\bQ}|_{U\cap Y_{k-l-1}}) \cong \bH^k(U,\underline{\bQ})$ and, by Theorem \ref{thm:geometricleray}, we obtain $P_lH^k(U,\bQ) = 0$. It follows that $\Gr_l^PH^m(U,\bQ) = P_lH^k(U,\bQ)/P_{l-1}H^k(U,\bQ) = 0$ for all $l < k$.

Similarly, if $l > k+m$ we have $k-l-1 < -m - 1$ and so $Y_{k-l} = Y_{k-l-1} = \emptyset$. Thus by Theorem \ref{thm:geometricleray} we have
\[P_lH^k(U,\bQ) = P_{l-1}H^k(U,\bQ) = \ker\left\{\bH^k(U,\underline{\bQ}) \longrightarrow 0\right\} = H^k(U,\bQ).\] 
It follows that $\Gr_l^PH^m(U,\bQ) = P_lH^k(U,\bQ)/P_{l-1}H^k(U,\bQ) = 0$.

The proofs for compactly supported cohomology follow by an analogous argument, or by applying Poincar\'{e}-Verdier duality to the result above.
\end{proof}

Using this, we prove a useful lemma relating the perverse Leray filtration on $Y$ to the perverse Leray filtrations on the subvarieties $Z_{-r}$.

\begin{lemma}\label{lem:perversesubvarieties} Let $i_r\colon Z_{-r} \hookrightarrow Y$ denote the inclusion.
\begin{enumerate}
\item If $l < k$, the morphism $(i_r)_!$ of mixed Hodge structures is an isomorphism for all $0\leq r < k-l$
\[\Gr^P_{l-r}H^{k-2r}(Z_{-r},\bQ) \cong \Gr^P_lH^k(Y,\bQ)(r) \]
and is surjective if $r = k-l$ 
\[\Gr^P_{2l-k}H^{2l-k}(Z_{l-k},\bQ) \twoheadrightarrow \Gr^P_lH^k(Y,\bQ)(k-l),\]
where $(-)$ denotes the usual Tate twist.
\item If $l > k$, the morphism $i_r^*$ of mixed Hodge structures is an isomorphism for all $0 \leq r < l-k$ 
\[\Gr^P_lH^k(Y,\bQ) \cong \Gr^P_{l-r}H^k(Z_{-r},\bQ)\]
and is injective if $r = l-k$
\[\Gr^P_lH^k(Y,\bQ) \hookrightarrow \Gr^P_{k}H^k(Z_{k-l},\bQ).\]
\end{enumerate}
\end{lemma}

\begin{proof}
To prove statement (1), assume first that $l < k-1$. By Lemma \ref{lem:Uvanishing} we have $\Gr_l^PH^k(U,\bQ) = \Gr_l^PH^{k-1}(U,\bQ) = 0$. Applying the first exact sequence of Proposition \ref{prop:perverselocalization}, we obtain an isomorphism 
\[(i_{1})_!\colon \Gr_{l-1}^PH^{k-2}(Z,\bQ) \stackrel{\sim}{\longrightarrow} \Gr_{l}^PH^{k}(Y,\bQ)(1).\] 
Replacing $Y$ by $Z_{-i}$ and $Z$ by $Z_{-i-1}$, we may repeat this argument inductively $k-l-1$ times; the resulting composition gives the required isomorphisms
\[\Gr^P_{l-r}H^{k-2r}(Z_{-r},\bQ) \cong \Gr^P_lH^k(Y,\bQ)(r)\]
for all $0 \leq r < l-k$. 

Now if $l = k-1$, by Lemma \ref{lem:Uvanishing} again we have $\Gr_{k-1}^PH^k(U,\bQ) = 0$, and the first exact sequence of Proposition \ref{prop:perverselocalization} gives a surjective map 
\[(i_{1})_!\colon\Gr_{k-2}^PH^{k-2}(Z,\bQ) \twoheadrightarrow \Gr_{k-1}^PH^{k}(Y,\bQ)(1).\] 
Replacing $Y$ by $Z_{l-k+1}$ and $Z$ by $Z_{l-k}$ and composing with the isomorphism above gives the required surjective map.

Statement (2) is proved similarly, using the vanishing of $\Gr_l^PH^k_c(U,\bQ)$ for $l > k$ from Lemma \ref{lem:Uvanishing} and the second exact sequence from Proposition \ref{prop:perverselocalization}.
\end{proof}

Now we prove a vanishing result for the graded pieces of the perverse Leray filtration on $Y$.

\begin{lemma}\label{lem:Yvanishing}
Recall that $m := \dim_{\bC}(B)$. Then we have $\Gr_l^PH^k(Y,\bQ) = 0$ whenever $l < \max\{\frac{k}{2},\, k - m\}$ or $l > k+m$.
\end{lemma}

\begin{proof}
This is a straightforward consequence of Lemma \ref{lem:perversesubvarieties}. 

The surjective map from Lemma \ref{lem:perversesubvarieties}(1) implies that if $l < k$, then  $\Gr^P_lH^k(Y,\bQ)$ vanishes whenever $\Gr^P_{2l-k}H^{2l-k}(Z_{l-k},\bQ)$ does. This occurs if $l < \frac{k}{2}$, as $Z_{l-k}$ has no cohomology in negative degrees, and if  $l < k-m$, as in this case $Z_{l-k} = \emptyset$.

The injective map of Lemma \ref{lem:perversesubvarieties}(2) implies that if $l > k$, then $\Gr^P_lH^k(Y,\bQ)$ vanishes whenever $\Gr^P_{k}H^k(Z_{k-l},\bQ)$ does. This occurs if $l > k + m$, as in this case $Z_{k-l} = \emptyset$.
\end{proof}

Finally, the following lemma is helpful when computing perverse Leray filtrations.

\begin{lemma}\label{lem:lerayisomorphisms}
Recall that $n := \dim_{\bC}(Y)$. There are isomorphisms of mixed Hodge structures
\[
\begin{tikzcd}
\Gr_l^PH^k(Y,\bQ) \ar[d,equal] \ar[r,"\sim"]& \Gr_{2n-l}^PH^{2n+k-2l}(Y,\bQ)(n-l)\ar[d,equal]\\
\Gr^P_{2n-l}H^{2n-k}(Y,\bQ)(n-k) &   \Gr_l^PH^{2l-k}(Y,\bQ)(l-k) \ar[l,"\sim"']
\end{tikzcd}\]
where $(-)$ denotes the usual Tate twist. Similar statements hold for $U$, if the top row is replaced by compactly supported cohomology.
\end{lemma}
\begin{proof} The vertical isomorphisms are just Poincar\'{e}-Verdier duality. The horizontal isomorphisms follow from an application of the Hard Lefschetz Theorem \cite[Theorem 2.1.4]{htam} to the definition of the perverse Leray filtration.
\end{proof}

\subsection{The cup-product filtration and the perverse Leray filtration}\label{sec:cupproduct}

We conclude this section by describing the relationship between the perverse Leray filtration and a second filtration, called the \emph{cup-product filtration}. As we shall see, the cup product filtration will give us an easy way to compute the perverse Leray filtration.

With notation as above, let $[Z] \in H^2(Y,\bQ)$ (resp. $[Z_{-r}] \in H^{2r}(Y,\bQ)$) denote the class of $Z$ (resp. $Z_{-r}$). Recall that $n := \dim_{\bC}(Y)$ and $m := \dim_{\bC}(B)$.

\begin{definition}
The cup product $[Z]\cup -$ defines a nilpotent endomorphism with index $(m+1)$ on the total cohomology $H^*(Y,\bQ)$. By \cite[Lemma-Definition 11.9]{mhs}, it therefore defines an increasing weight filtration on $H^*(Y,\bQ)$, centred at $n$. We call this filtration the \emph{cup-product filtration} and denote it by $K_{\bullet}$.
\end{definition}

Recall that the cup-product with the class of $Z$ may be expressed as the composition $[Z]\cup - = i^*i_!$, where $i\colon Z \hookrightarrow Y$ is the inclusion. Using Proposition \ref{prop:perverselocalization}, we see that this induces a map on the graded pieces of the perverse Leray filtration as follows.

\[
\begin{tikzcd}\Gr_l^PH^k(Y,\bQ) \ar[rr,"{[Z]\cup -}"] \ar[rd,"{i^*}"]&& \Gr_l^PH^{k+2}(Y,\bQ)(1)\\
& \Gr_{l-1}^PH^k(Z,\bQ) \ar[ur,"{i_!}"] &
\end{tikzcd}\]

Iterating this map $r$-times and noting that the $r$-fold cup-product of $[Z]$ with itself is $[Z_{-r}]$, we obtain a morphism of mixed Hodge structures
\[[Z_{-r}] \cup - : \Gr_l^PH^k(Y,\bQ) \longrightarrow \Gr_l^PH^{k+2r}(Y,\bQ)(r).\]
%Setting $k = l-r$, we obtain as a  special case
%\[[Z_{-r}] \cup - : \Gr_l^PH^{l-r}(Y,\bQ) \longrightarrow \Gr_l^PH^{l+r}(Y,\bQ)(r).\]

In the case $k = l-r$, the Hard Lefschetz Theorem for Perverse Cohomology Groups \cite[Theorem 2.1.4]{htam} tells us this this is an isomorphism.

\begin{proposition}\cite[Theorem 2.1.4]{htam} \label{prop:cupproduct} The cup-product map
\[[Z_{-r}] \cup - \colon \Gr_l^PH^{l-r}(Y,\bQ) \longrightarrow \Gr_l^PH^{l+r}(Y,\bQ)(r)\]
is an isomorphism of mixed Hodge structures for all $0 < r \leq m:= \dim_{\bC}(B)$ and all $l \geq r$.
\end{proposition}

Using this result, we can give a particularly simple interpretation of the perverse Leray filtration on $Y$. 

\begin{proposition}\label{prop:cupperverse}
Recall that $n := \dim_{\bC}(Y)$. There are isomorphisms between the graded pieces of the cup-product filtration  and the shifted perverse Leray filtration,
\[\Gr_l^KH^k(Y,\bQ) \cong \Gr^P_{l + k - n}H^{k}(Y,\bQ).\]
\end{proposition}
\begin{proof} This is essentially \cite[Proposition 5.2.4]{htam}, but some effort is required to extract the statement as written above. As the proof is short we include it for clarity.

By \cite[Lemma-Definition 11.9]{mhs}, the cup-product filtration is uniquely defined by the two properties
\begin{enumerate}
\item $[Z] \cup -$ maps $K_i$ into $K_{i-2}$ for all $i \geq 2$, and
\item the map
\[[Z_{-r}] \cup -\colon \Gr^K_{n+r}H^*(Y,\bQ) \longrightarrow \Gr^K_{n-r}H^*(Y,\bQ)\]
is an isomorphism for all $0 < r \leq m = \dim_{\bC}(B)$.
\end{enumerate}

To prove that $K_{\bullet}$ coincides with the shifted perverse Leray filtration, it suffices to show that the shifted perverse Leray filtration satisfies these two properties. Both are easily verified from the discussion above.
%Indeed, from the discussion above, we have
%\[[Z] \cup - : \Gr^P_{l + k - n}H^k(Y,\bQ) \longrightarrow \Gr^P_{l+(k+2) - n - 2}H^{k+2}(Y,\bQ),\] 
%proving (1), and (2) is precisely Proposition \ref{prop:cupproduct}.
\end{proof}

\section{The mirror Clemens-Schmid sequence}\label{sec:mirrorCS}

We next show how we can combine the sequences from Proposition \ref{prop:perverselocalization} into a four-term sequence. Indeed, consider the diagram
\[\begin{tikzcd}[column sep=small,cramped]H^{k-1}_c(U) \ar[r] & H^{k-1}(Y) \ar[rr,dotted,"{[Z]\cup -}"] \ar[rd] &[-2.75em]&[-2.75em] H^{k+1}(Y)(1) \ar[r]  & H^{k+1}(U)(1) \ar[rr,dotted] \ar[rd]&[-2.75em]&[-2.75em] H^{k+1}_c(U)  \\ 
 & & H^{k-1}(Z)[-1] \ar[ur]\ar[rd] &  && H^{k}(Z)[-1] \ar[ur] \ar[rd] &\\
H^{k}(Y)(1) \ar[r]& H^{k}(U)(1) \ar[rr,dotted] \ar[ur] && H^{k}_c(U) \ar[r] & H^{k}(Y) \ar[rr,dotted,"{[Z]\cup -} "] \ar[ur]&& H^{k+2}(Y)(1)
\end{tikzcd}\]
obtained by interlacing the exact sequences from Proposition \ref{prop:perverselocalization}. The morphisms are all morphisms of mixed Hodge structures on the graded pieces of the perverse Leray filtration; the dotted arrows denote the morphisms induced by composition with the maps through $H^k(Z,\bQ)[-1]$. Note that the induced map $H^{k-1}(Y,\bQ) \to H^{k+1}(Y,\bQ)$ is given by cup-product with the class of $Z$ (see Subsection \ref{sec:cupproduct}). As usual, $[-1]$ denotes a shift in the perverse Leray filtration and $(1)$ denotes the Tate twist.

To prove Theorem \ref{thm:mainthm}, it remains to show the following.

\begin{proposition} \label{prop:mirrorcs}
The top and bottom rows of this diagram are exact sequences of mixed Hodge structures on the graded pieces of the perverse Leray filtration. 
\end{proposition}

\begin{proof}
As all of the morphisms respect the perverse Leray filtration, it suffices to show that the sequence is exact on the graded pieces of this filtration. We begin by proving that the subsequence
\begin{equation}\label{eq:localinvariantcycles} \Gr_l^PH^{k-1}_c(U) \longrightarrow \Gr_l^PH^{k-1}(Y) \stackrel{[Z]\cup -}{\longrightarrow} \Gr_l^PH^{k+1}(Y)\longrightarrow \Gr_l^PH^{k+1}(U)
\end{equation}
is exact at the two middle terms.

If $l < k$, by Lemma \ref{lem:Uvanishing} we have $\Gr_l^PH^k(U)  = \Gr_l^PH^{k+1}(U) = 0$, giving the diagram
\[\begin{tikzcd}\Gr_l^PH^{k-1}_c(U) \ar[r] & \Gr_l^PH^{k-1}(Y) \ar[rr,"{[Z]\cup -}"] \ar[rd] &[-2em]&[-2em] \Gr_l^PH^{k+1}(Y) \ar[r]  & 0.  \\ 
 & & \Gr_{l-1}^PH^{k-1}(Z) \ar[ur,equal] &&  
\end{tikzcd}\]
In this case it suffices to prove that the cup-product map is surjective. If $l \leq k-m$, for $m = \dim_{\bC}(B)$, Lemma \ref{lem:Yvanishing} gives $\Gr_l^PH^{k+1}(Y) = 0$ and this is immediate. If $k-m < l < k$, surjectivity follows from the fact that, by Proposition \ref{prop:cupproduct}, the composition
\[\begin{tikzcd}[column sep=large]\Gr_l^PH^{2l-k-1}(Y) \ar[r,"{[Z_{l-k}]\cup -}"] &  \Gr_l^PH^{k-1}(Y) \ar[r,"{[Z]\cup -}"] & \Gr_l^PH^{k+1}(Y)\end{tikzcd}\]
is an isomorphism.

If $l > k$, applying Lemma \ref{lem:Uvanishing} gives $\Gr_l^PH_c^{k-1}(U)  = \Gr_l^PH^{k}_c(U) = 0$, so we obtain the diagram
\[\begin{tikzcd}0 \ar[r] & \Gr_l^PH^{k-1}(Y) \ar[rr,"{[Z]\cup -}"] \ar[rd,equal] &[-2em]&[-2em] \Gr_l^PH^{k+1}(Y) \ar[r]  & \Gr_l^PH^{k+1}(U).  \\ 
& &  \Gr_{l-1}^PH^{k-1}(Z) \ar[ur] &&  
\end{tikzcd}\]
In this case it suffices to prove that the cup-product map is injective. If $l \geq k+m$, Lemma \ref{lem:Yvanishing} gives $\Gr_l^PH^{k-1}(Y) = 0$ and this is immediate. If $k < l < k+m$, injectivity follows from the fact that, by Proposition \ref{prop:cupproduct}, the composition
\[\begin{tikzcd}[column sep=large] \Gr_l^PH^{k-1}(Y) \ar[r,"{[Z]\cup -}"] & \Gr_l^PH^{k+1}(Y)\ar[r,"{[Z_{k-l}]\cup -}"] &  \Gr_l^PH^{2l-k+1}(Y) \end{tikzcd}\]
is an isomorphism.

Finally, if $l = k$, by Lemma \ref{lem:Uvanishing} we obtain $\Gr_l^PH_c^{k-1}(U)  = \Gr_l^PH^{k+1}(U) = 0$ and our diagram becomes
\[\begin{tikzcd} 0 \ar[r] & \Gr_l^PH^{k-1}(Y) \ar[rr,"{[Z]\cup -}"] \ar[rd] &[-2em]&[-2em] \Gr_l^PH^{k+1}(Y) \ar[r]  & 0.  \\ 
& &  \Gr_{l-1}^PH^{k-1}(Z) \ar[ur] && 
\end{tikzcd}\]
In this case it suffices to prove that the cup-product map is an isomorphism, but this is precisely the $r=1$ case of Proposition \ref{prop:cupproduct}.

With this complete, the proof of the theorem follows by a straightforward diagram chase.
\end{proof}

\begin{remark} In the appendix to \cite{htdicdt}, Saito shows that exactness of the classical Clemens-Schmid sequence is equivalent to both the \emph{Decomposition Theorem} and the \emph{Local Invariant Cycle Property}, and that these properties may be formulated on the level of the derived categories \cite[Section A.4]{htdicdt}. It should be possible to extract our result above by similar methods, albeit in a different category to the one considered by Saito: in this case the Decomposition Theorem is the result of the same name for perverse sheaves proved by Beilinson, Bernstein, and Deligne \cite{fp} (see also \cite[Theorem 2.1.1]{htam}), and the Local Invariant Cycle Property is essentially exactness of the sequence \eqref{eq:localinvariantcycles}.
\end{remark}

\begin{remark} It is possible to interlace another well-known long exact sequence of mixed Hodge structures with the two localization sequences and the mirror Clemens-Schmid sequence, as follows. Indeed, let $\varphi\colon E \to Y$ denote the total space of the $\bC^*$-bundle associated to the line bundle $\calO_Y(Z)$, equipped with the usual mixed Hodge structure on an open variety; note that the Euler class of $E$ is $[Z]$. Then we have the Gysin sequence 
\[\cdots \to H^k(Y,\bQ)(1)  \stackrel{\varphi^*}{\longrightarrow} H^k(E,\bQ)(1) \stackrel{\varphi_!}{\longrightarrow} H^{k-1}(Y,\bQ) \stackrel{[Z]\cup -}{\longrightarrow}  H^{k+1}(Y,\bQ)(1) \to \cdots,\]
which interlaces with our other sequences as follows,
\[\begin{tikzcd}[column sep=small,cramped]H^{k-1}_c(U) \ar[rr] &[-2.75em]&[-2.75em] H^{k-1}(Y) \ar[rr,"{[Z]\cup -}"] \ar[rd] &[-2.5em]&[-2.5em] H^{k+1}(Y)(1) \ar[rr] \ar[dr] &[-2.75em]&[-2.75em] H^{k+1}(U)(1) \ar[rr] \ar[rd]&[-2.5em]&[-2.5em] H^{k+1}_c(U)  \\ 
 &H^k(E)(1) \ar[ur]& & H^{k-1}(Z)\ar[ur]\ar[rd] & & H^{k+1}(E)(1)\ar[dr] && H^{k}(Z) \ar[ur] \ar[rd] &\\
H^{k}(Y)(1) \ar[rr] \ar[ur]&& H^{k}(U)(1) \ar[rr] \ar[ur] && H^{k}_c(U) \ar[rr] && H^{k}(Y) \ar[rr,"{[Z]\cup -} "] \ar[ur]&& H^{k+2}(Y)(1).
\end{tikzcd}\]
A simple diagram chase, using Propositon \ref{prop:mirrorcs}, verifies that everything in this diagram commutes. Note that we have suppressed the perverse shifts: it is not particularly clear how to define a perverse filtration on the cohomology of $E$ in a way that is compatible with the rest of the picture.
\end{remark}

\section{Mirror symmetry}\label{sec:mirror}

\subsection{The Clemens-Schmid Sequence}

We call the four-term exact sequence constructed in the previous section the \emph{mirror Clemens-Schmid sequence}. To explain our choice of nomenclature and the relation of these results to mirror symmetry, we first take a closer look at the usual Clemens-Schmid sequence. We begin by recalling the standard setup; for more detail the reader may consult \cite[Chapter 11]{mhs} or \cite{csesa}.

For a semistable degeneration $\calX \to \Delta$ (see Definition \ref{def:degeneration}), the Clemens-Schmid sequence is the four-term exact sequence of mixed Hodge structures
\[\cdots\to H^k(\calX) \stackrel{i^*}{\longrightarrow} H^k_{\lim}(X) \stackrel{\nu}{\longrightarrow} H^k_{\lim}(X)(-1) \longrightarrow H^{k+2}_{X_0}(\calX) \longrightarrow H^{k+2}(\calX) \to\cdots\]
where
\begin{itemize}
\item $H^k_{\lim}(X)$ is the usual cohomology of $X$ equipped with the limiting mixed Hodge structure;
\item $H^k(\calX) \cong H^k(X_0)$ is equipped with Deligne's mixed Hodge structure on a normal crossing variety;
\item $H^{k+2}_{X_0}(\calX)$ is Poincar\'{e} dual to $H_{2n-k}(\calX)$, which is given a mixed Hodge structure by identification with $\Hom(H^{2n-k}(\calX),\bQ)$;
\item $\nu\colon H^k_{\lim}(X) \to H^k_{\lim}(X)(-1)$ is the morphism of mixed Hodge structures induced by the logarithm of monodromy;
\item $i\colon X \hookrightarrow \calX$ is the inclusion.
\end{itemize}

To discuss mirror symmetry, we would like to place a perverse Leray filtration on the terms of this sequence, corresponding to the map $\calX \to \Delta$. For the cohomology of $X$ this is straightforward: as $X$ is fibred over a point, the perverse Leray filtration is trivial and concentrated in middle degree; i.e. we have $\Gr^P_kH^k(X,\bQ) = H^k(X,\bQ)$ and all of the other graded pieces vanish.

For $\calX$, however, whilst one may define the perverse Leray filtration to be the filtration induced by the perverse Leray spectral sequence in the same way as before, most of the results that we have used to study it do not hold in this setting, as the base $\Delta$ of the fibration is not an algebraic variety.  Instead we apply results of Kerr and Laza \cite{htdicdt}, who have recently studied the interaction between the perverse Leray filtration and the Clemens-Schmid sequence.

With our indexing convention, the results of \cite[Section 10]{htdicdt} give a simple description of the perverse Leray filtration on the cohomology of $\calX$.
%\begin{align*}
%P_{k-1}H^k(\calX,\bQ) &=  0  \\
%P_kH^k(\calX,\bQ) &=  \ker\{i^*\colon H^k(\calX,\bQ) \to H^k(X,\bQ)\} \\
%P_{k+1}H^k(\calX,\bQ) &= H^k(\calX,\bQ).
%\end{align*}
The nonzero graded pieces are
\begin{align*}
\Gr^P_{k}H^k(\calX,\bQ) &= \ker\{i^*\colon H^k(\calX,\bQ) \to H^k(X,\bQ)\}\\
\Gr^P_{k+1}H^k(\calX,\bQ) &= \coim\{i^*\colon H^k(\calX,\bQ) \to H^k(X,\bQ)\}.
\end{align*}
Note that the filtrand $P_kH^k(\calX,\bQ) = \ker(i^*)$ is called the \emph{phantom cohomology} by Kerr and Laza.  By the discussion in \cite[Section 4]{htdicdt}, each of these graded pieces admits a natural mixed Hodge structure induced by Deligne's mixed Hodge structure on $H^k(\calX,\bQ) \cong H^k(X_0,\bQ)$.

Finally, to compute the perverse Leray filtration on $H^{k}_{X_0}(\calX,\bQ)$ we apply Poincar\'{e} duality (see \cite[Section 4]{htdicdt}), noting that $i^*$ is dual to the map $H^{k-2}_{\lim}(X) \to H^{k}_{X_0}(\calX)$ from the Clemens-Schmid sequence. 
%We obtain the filtration
%\begin{align*}
%P_{k-2}H^{k}_{X_0}(\calX,\bQ) &=  0  \\
%P_{k-1}H^{k}_{X_0}(\calX,\bQ) &=  \im\{H^{k-2}_{\lim}(X,\bQ) \to H^{k}_{X_0}(\calX,\bQ)\} \\
%P_{k}H^{k}_{X_0}(\calX,\bQ) &= H^{k}_{X_0}(\calX,\bQ).
%\end{align*}
The nonzero graded pieces are
\begin{align*}
\Gr^P_{k-1}H^{k}_{X_0}(\calX,\bQ) &=  \im\{H^{k-2}_{\lim}(X,\bQ) \to H^{k}_{X_0}(\calX,\bQ)\} \\
\Gr^P_{k}H^{k}_{X_0}(\calX,\bQ) &=  \coker\{H^{k-2}_{\lim}(X,\bQ) \to H^{k}_{X_0}(\calX,\bQ)\}.
\end{align*}
It follows from the equivalent statement for $H^k(\calX,\bQ)$ that the mixed Hodge structure on $H^{k}_{X_0}(\calX,\bQ)$ induces mixed Hodge structures on the graded pieces of $P_{\bullet}$. The following result is immediate from these considerations.

\begin{proposition}\label{prop:perverseclemensschmid} The Clemens-Schmid sequence
\[\to H^k(\calX)[1] \stackrel{i^*}{\longrightarrow} H^k_{\lim}(X) \stackrel{\nu}{\longrightarrow} H^k_{\lim}(X)(-1) \longrightarrow H^{k+2}_{X_0}(\calX)[1] \longrightarrow H^{k+2}(\calX)[1] \to\]
splits into exact sequences of mixed Hodge structures on the graded pieces of the perverse Leray filtration $P_{\bullet}$. Here $[1]$ denotes a shift in the perverse Leray filtration and $(-1)$ denotes the usual Tate twist on the mixed Hodge structure.
\end{proposition}

%\begin{proof}
%It suffices to show exactness on the graded pieces of $P_{\bullet}$. Taking graded pieces, the sequence splits into subsequences
%\begin{align*}
%\Gr_{k-1}^P\colon&\quad 0 \to \coker\{H^{k-2}_{\lim}(X) \to H^{k}_{X_0}(\calX)\} \to \ker(i^*) \to 0,\\
%\Gr_{k}^P\colon& \quad 0 \to \coim(i^*) \to H^k_{\lim}(X) \to H^k_{\lim}(X) \to \im\{H^k_{\lim}(X) \to H^{k+2}_{X_0}(\calX)\} \to 0.
%\end{align*}
%Exactness then follows immediately from exactness of the Clemens-Schmid sequence.
%\end{proof}

\subsection{A mirror conjecture}

Suppose now that $\calX \to \Delta$ is a $K$-trivial semistable degeneration of Calabi-Yau varieties of length $\lambda$ (see Definitions \ref{def:length} and \ref{def:Ktrivial}). Following the general philosophy of \cite{mstdfcym}, one might expect such a degeneration to be mirror to a fibration on the mirror Calabi-Yau variety $\pi\colon Y \to B$, with $\dim_{\bC}(B) = \lambda - 1$. The map $\nu$ is mirror to the cup product
\[[Z] \cup - : H^k(Y,\bQ) \longrightarrow H^{k+2}(Y,\bQ)(1);\]
where $\nu$ moves us horizontally across the Hodge diamond, $[Z] \cup -$ moves us downwards.

Under such a correspondence, we claim that the Clemens-Schmid sequence and the mirror Clemens-Schmid sequence should be thought of as mirror to one another. Comparing terms in these two sequences, noting that mirror symmetry is expected to exchange Hodge numbers and exchange the weight and perverse Leray filtrations, gives rise to Conjecture \ref{con:mirrorfiltrations}.

\begin{remark} The unexpected $+1$ in the $P_{\bullet}$ index of 
\[\Gr_F^{n-p}\Gr^W_{n+l-2p}\Gr^P_{n+q-2p+1}H^{n+k-2p}(\calX)\]
in Conjecture \ref{con:mirrorfiltrations} compensates for the fact that $\calX$ has complex dimension $n + 1$, which shifts the perverse Leray filtration $P_{\bullet}$ up by one relative to the cohomology of $X$, $Y$, and $U$, all of which have dimension $n$. This phenomenon also appears in the localization sequences (Proposition \ref{prop:perverselocalization}) and the Clemens-Schmid sequence (Proposition \ref{prop:perverseclemensschmid}), which all contain shifts in the perverse Leray filtration to compensate for differing dimensions.
\end{remark}

\section{K3 surfaces}\label{sec:K3}

In this section we will prove Theorem \ref{thm:K3evidence}, which provides some evidence for Conjecture \ref{con:mirrorfiltrations} in the K3 surface setting.

\subsection{Fibrations on K3 surfaces}\label{sec:K3fibex}

We first compute the various filtrations from Section \ref{sec:perverse} in the case where $Y$ is a smooth K3 surface. 

\subsubsection{$B$ is a curve} \label{sec:Bcurve} Begin by assuming that the base $B$ is a curve. Then necessarily $B \cong \bP^1$ and the fibres of $\pi$ are generically genus one curves. Rather than taking the trivial embedding of $B$ into $\bP^1$, we instead take the $r$-fold Veronese embedding, so that $C$ is a union of $r$ points and $Z$ is a union of $r$ disjoint genus one fibres in $Y$.

With this setup it is easy to compute the following mixed Hodge structures $(W_{\bullet},F^{\bullet})$ on the graded pieces of the perverse Leray filtration $P_{\bullet}$, using Propositions \ref{prop:perverselocalization} and \ref{prop:cupperverse}. The entries in the tables show the dimensions of the graded pieces $\Gr_F^p\Gr^W_q\Gr^P_lH^k$. All other graded pieces of $P_{\bullet}$ are zero.

For the cohomology of $Y$ we have the following filtrations.

\begin{center}
\begin{TAB}{|c|c|c|c|c|c|c|}{|c|t|}
$\Gr_1^PH^0(Y)$&&$\Gr_1^PH^2(Y)$&$\Gr_2^PH^2(Y)$&$\Gr_3^PH^2(Y)$ && $\Gr_3^PH^4(Y)$\\
\begin{tabular}{c|c}
\diagbox[innerleftsep=1pt,innerrightsep=1pt,width=18pt,height=18pt]{\scriptsize{$F$}}{\scriptsize{$W$}} & $0$\\
\hline
$0$& $1$
\end{tabular}&&
\begin{tabular}{c|c}
\diagbox[innerleftsep=1pt,innerrightsep=1pt,width=18pt,height=18pt]{\scriptsize{$F$}}{\scriptsize{$W$}} & $2$\\
\hline
$1$& $1$
\end{tabular}&
\begin{tabular}{c|c}
\diagbox[innerleftsep=1pt,innerrightsep=1pt,width=18pt,height=18pt]{\scriptsize{$F$}}{\scriptsize{$W$}} & $2$\\
\hline
$0$& $1$\\
$1$ & $18$\\
$2$ & $1$
\end{tabular}&
\begin{tabular}{c|c}
\diagbox[innerleftsep=1pt,innerrightsep=1pt,width=18pt,height=18pt]{\scriptsize{$F$}}{\scriptsize{$W$}} & $2$\\
\hline
$1$& $1$
\end{tabular}
&&
\begin{tabular}{c|c}
\diagbox[innerleftsep=1pt,innerrightsep=1pt,width=18pt,height=18pt]{\scriptsize{$F$}}{\scriptsize{$W$}} & $4$\\
\hline
$2$& $1$
\end{tabular}
\end{TAB}
\end{center}

%For the cohomology of $Z$ we have the following filtrations.
%
%\begin{center}
%\begin{TAB}{|c|c|c|c|c|}{|c|t|}
%$\Gr_0^PH^0(Z)$&&$\Gr_1^PH^1(Z)$&& $\Gr_2^PH^2(Z)$\\
%\begin{tabular}{c|c}
%\diagbox[innerleftsep=1pt,innerrightsep=1pt,width=18pt,height=18pt]{\scriptsize{$F$}}{\scriptsize{$W$}} & $0$\\
%\hline
%$0$& $r$
%\end{tabular}&&
%\begin{tabular}{c|c}
%\diagbox[innerleftsep=1pt,innerrightsep=1pt,width=18pt,height=18pt]{\scriptsize{$F$}}{\scriptsize{$W$}} & $1$\\
%\hline
%$0$& $r$\\
%$1$ & $r$
%\end{tabular}
%&&
%\begin{tabular}{c|c}
%\diagbox[innerleftsep=1pt,innerrightsep=1pt,width=18pt,height=18pt]{\scriptsize{$F$}}{\scriptsize{$W$}} & $2$\\
%\hline
%$1$& $r$
%\end{tabular}
%\end{TAB}
%\end{center}

For the compactly supported cohomology of $U$ we have the following filtrations. The filtrations on the cohomology of $U$ can be obtained from these by Poincar\'{e}-Verdier duality.

\begin{center}
\begin{TAB}{|c|c|c|c|c|c|c|c|}{|c|t|}
$\Gr_1^PH^1_c(U)$&&$\Gr_1^PH^2_c(U)$&$\Gr_2^PH^2_c(U)$&&$\Gr_3^PH^3_c(U)$&& $\Gr_3^PH^4_c(U)$\\
\begin{tabular}{c|c}
\diagbox[innerleftsep=1pt,innerrightsep=1pt,width=18pt,height=18pt]{\scriptsize{$F$}}{\scriptsize{$W$}} & $0$\\
\hline
$0$& $r-1$
\end{tabular}&&
\begin{tabular}{c|c}
\diagbox[innerleftsep=1pt,innerrightsep=1pt,width=18pt,height=18pt]{\scriptsize{$F$}}{\scriptsize{$W$}} & $2$\\
\hline
$1$& $1$
\end{tabular}&
\begin{tabular}{c|cc}
\diagbox[innerleftsep=1pt,innerrightsep=1pt,width=18pt,height=18pt]{\scriptsize{$F$}}{\scriptsize{$W$}} & $1$&$2$\\
\hline
$0$& $r$ & $1$\\
$1$ & $r$ & $18$\\
$2$ & $0$ & $1$
\end{tabular}
&&
\begin{tabular}{c|c}
\diagbox[innerleftsep=1pt,innerrightsep=1pt,width=18pt,height=18pt]{\scriptsize{$F$}}{\scriptsize{$W$}} & $2$\\
\hline
$1$& $r-1$
\end{tabular}
&&
\begin{tabular}{c|c}
\diagbox[innerleftsep=1pt,innerrightsep=1pt,width=18pt,height=18pt]{\scriptsize{$F$}}{\scriptsize{$W$}} & $4$\\
\hline
$2$& $1$
\end{tabular}
\end{TAB}
\end{center}

\subsubsection{$B$ is a surface} \label{sec:Bsurface} Next assume that $B$ is a (possibly singular) surface. As the preimage of a general hyperplane section, the curve $Z \subset Y$ is smooth by Bertini's Theorem and connected by \cite[Proposition 1]{ctpvaism}. Let $g$ denote the genus of $Z$. With this setup, the graded pieces of the various filtrations are as follows, where we present the data in the same form as above.

For the cohomology of $Y$ we have the following filtrations.

\begin{center}
\begin{TAB}{|c|c|c|c|c|}{|c|t|}
$\Gr_2^PH^0(Y)$&&$\Gr_2^PH^2(Y)$ && $\Gr_2^PH^4(Y)$\\
\begin{tabular}{c|c}
\diagbox[innerleftsep=1pt,innerrightsep=1pt,width=18pt,height=18pt]{\scriptsize{$F$}}{\scriptsize{$W$}} & $0$\\
\hline
$0$& $1$
\end{tabular}&&
\begin{tabular}{c|c}
\diagbox[innerleftsep=1pt,innerrightsep=1pt,width=18pt,height=18pt]{\scriptsize{$F$}}{\scriptsize{$W$}} & $2$\\
\hline
$0$ & $1$\\
$1$& $20$\\
$2$ & $1$
\end{tabular}
&&
\begin{tabular}{c|c}
\diagbox[innerleftsep=1pt,innerrightsep=1pt,width=18pt,height=18pt]{\scriptsize{$F$}}{\scriptsize{$W$}} & $4$\\
\hline
$2$& $1$
\end{tabular}
\end{TAB}
\end{center}

%For the cohomology of $Z$ we have the following filtrations.
%
%\begin{center}
%\begin{TAB}{|c|c|c|c|c|}{|c|t|}
%$\Gr_1^PH^0(Z)$&&$\Gr_1^PH^1(Z)$&& $\Gr_1^PH^2(Z)$\\
%\begin{tabular}{c|c}
%\diagbox[innerleftsep=1pt,innerrightsep=1pt,width=18pt,height=18pt]{\scriptsize{$F$}}{\scriptsize{$W$}} & $0$\\
%\hline
%$0$& $1$
%\end{tabular}&&
%\begin{tabular}{c|c}
%\diagbox[innerleftsep=1pt,innerrightsep=1pt,width=18pt,height=18pt]{\scriptsize{$F$}}{\scriptsize{$W$}} & $1$\\
%\hline
%$0$& $g$\\
%$1$ & $g$
%\end{tabular}
%&&
%\begin{tabular}{c|c}
%\diagbox[innerleftsep=1pt,innerrightsep=1pt,width=18pt,height=18pt]{\scriptsize{$F$}}{\scriptsize{$W$}} & $2$\\
%\hline
%$1$& $1$
%\end{tabular}
%\end{TAB}
%\end{center}

For the compactly supported cohomology of $U$ we have the following filtrations. The filtrations on the cohomology of $U$ can be obtained from these by Poincar\'{e}-Verdier duality.

\begin{center}
\begin{TAB}{|c|c|c|}{|c|t|}
$\Gr_2^PH^2_c(U)$&& $\Gr_2^PH^4_c(U)$\\

\begin{tabular}{c|cc}
\diagbox[innerleftsep=1pt,innerrightsep=1pt,width=18pt,height=18pt]{\scriptsize{$F$}}{\scriptsize{$W$}} & $1$&$2$\\
\hline
$0$& $g$ & $1$\\
$1$ & $g$ & $19$\\
$2$ & $0$ & $1$
\end{tabular}
&&
\begin{tabular}{c|c}
\diagbox[innerleftsep=1pt,innerrightsep=1pt,width=18pt,height=18pt]{\scriptsize{$F$}}{\scriptsize{$W$}} & $4$\\
\hline
$2$& $1$
\end{tabular}
\end{TAB}
\end{center}

\subsection{Degenerations of K3 surfaces}\label{sec:K3degex}

We next exhibit the various filtrations from Section \ref{sec:mirror} in the case where $\calX \to \Delta$ is a $K$-trivial semistable degeneration of K3 surfaces. Recall, in particular, that such degenerations of lengths $1$, $2$, and $3$ correspond precisely to degenerations of Types I, II, and III, respectively.  A good summary of the geometry of such degenerations of K3 surfaces may be found in \cite[Section 6.5]{lok3s} and a more detailed discussion of the Clemens-Schmid sequence in this setting can be found in \cite[Section 4(d)]{csesa}.

\subsubsection{Type II degeneration} \label{sec:typeii} We begin by considering a Type II degeneration of K3 surfaces. The central fibre $X_0$ of such a degeneration is a chain of surfaces ruled over elliptic curves, glued along sections of the rulings, with a rational surface at each end of the chain. We assume that $X_0$ has $(r+1)$ components.

As in Subsection \ref{sec:K3fibex}, we describe below the mixed Hodge structures $(W_{\bullet},F^{\bullet})$ on the graded pieces of the filtration $P_{\bullet}$. The entries in the tables show the dimensions of the graded pieces $\Gr_F^p\Gr^W_q\Gr^P_lH^k$. All other graded pieces of $P_{\bullet}$ are zero.

For the cohomology of $X$, with its limiting mixed Hodge structure, the filtrations are as follows.

\begin{center}
\begin{TAB}{|c|c|c|c|c|}{|c|t|}
$\Gr_0^PH^0_{\lim}(X)$&&$\Gr_2^PH^2_{\lim}(X)$ && $\Gr_4^PH^4_{\lim}(X)$\\
\begin{tabular}{c|c}
\diagbox[innerleftsep=1pt,innerrightsep=1pt,width=18pt,height=18pt]{\scriptsize{$F$}}{\scriptsize{$W$}} & $0$\\
\hline
$0$& $1$
\end{tabular}&&
\begin{tabular}{c|ccc}
\diagbox[innerleftsep=1pt,innerrightsep=1pt,width=18pt,height=18pt]{\scriptsize{$F$}}{\scriptsize{$W$}} & $1$ & $2$ & $3$\\
\hline
$0$& $1$ & $0$ & $0$ \\
$1$ & $1$ & $18$ & $1$ \\
$2$ & $0$ & $0$ & $1$
\end{tabular}
&&
\begin{tabular}{c|c}
\diagbox[innerleftsep=1pt,innerrightsep=1pt,width=18pt,height=18pt]{\scriptsize{$F$}}{\scriptsize{$W$}} & $4$\\
\hline
$2$& $1$
\end{tabular}
\end{TAB}
\end{center}

For the cohomology of $\calX$, with Deligne's mixed Hodge structure, the filtrations are as follows. The filtrations on $H^{k}_{X_0}(\calX)$ may be obtained from these by Poincar\'{e} duality.

\begin{center}
\begin{TAB}{|c|c|c|c|}{|c|t|}
$\Gr_1^PH^0(\calX)$&&$\Gr_2^PH^2(\calX)$&$\Gr_3^PH^2(\calX)$\\
\begin{tabular}{c|c}
\diagbox[innerleftsep=1pt,innerrightsep=1pt,width=18pt,height=18pt]{\scriptsize{$F$}}{\scriptsize{$W$}} & $0$\\
\hline
$0$& $1$
\end{tabular}&&
\begin{tabular}{c|c}
\diagbox[innerleftsep=1pt,innerrightsep=1pt,width=18pt,height=18pt]{\scriptsize{$F$}}{\scriptsize{$W$}} & $2$\\
\hline
$1$& $r$
\end{tabular}&
\begin{tabular}{c|cc}
\diagbox[innerleftsep=1pt,innerrightsep=1pt,width=18pt,height=18pt]{\scriptsize{$F$}}{\scriptsize{$W$}} & $1$ & $2$\\
\hline
$0$& $1$ & $0$\\
$1$ & $1$ & $18$
\end{tabular}
\end{TAB}
\smallskip

\begin{TAB}{|c|c|c|c|}{|c|t|}
$\Gr_3^PH^3(\calX)$&&$\Gr_4^PH^4(\calX)$ & $\Gr_5^PH^4(\calX)$\\
\begin{tabular}{c|c}
\diagbox[innerleftsep=1pt,innerrightsep=1pt,width=18pt,height=18pt]{\scriptsize{$F$}}{\scriptsize{$W$}} & $3$\\
\hline
$1$& $r-1$\\
$2$ & $r-1$
\end{tabular}
&&
\begin{tabular}{c|c}
\diagbox[innerleftsep=1pt,innerrightsep=1pt,width=18pt,height=18pt]{\scriptsize{$F$}}{\scriptsize{$W$}} & $4$\\
\hline
$2$& $r$
\end{tabular}
&
\begin{tabular}{c|c}
\diagbox[innerleftsep=1pt,innerrightsep=1pt,width=18pt,height=18pt]{\scriptsize{$F$}}{\scriptsize{$W$}} & $4$\\
\hline
$2$& $1$
\end{tabular}
\end{TAB}
\end{center}

\subsubsection{Type III degeneration} \label{sec:typeiii} Next we consider Type III degenerations. The central fibre $X_0$ of such a degeneration is a union of rational surfaces glued along anticanonical cycles, so that the dual graph of the central fibre is a sphere. We assume that $X_0$ has $(g+1)$ components. 

For the cohomology of $X$, with its limiting mixed Hodge structure, the filtrations are as follows.

\begin{center}
\begin{TAB}{|c|c|c|c|c|}{|c|t|}
$\Gr_0^PH^0_{\lim}(X)$&&$\Gr_2^PH^2_{\lim}(X)$ && $\Gr_4^PH^4_{\lim}(X)$\\
\begin{tabular}{c|c}
\diagbox[innerleftsep=1pt,innerrightsep=1pt,width=18pt,height=18pt]{\scriptsize{$F$}}{\scriptsize{$W$}} & $0$\\
\hline
$0$& $1$
\end{tabular}&&
\begin{tabular}{c|ccccc}
\diagbox[innerleftsep=1pt,innerrightsep=1pt,width=18pt,height=18pt]{\scriptsize{$F$}}{\scriptsize{$W$}} &$0$ & $1$ & $2$ & $3$ & $4$\\
\hline
$0$& $1$ & $0$& $0$ & $0$& $0$ \\
$1$ & $0$ &$0$& $20$ & $0$& $0$ \\
$2$ & $0$ & $0$ &$0$&$0$& $1$
\end{tabular}
&&
\begin{tabular}{c|c}
\diagbox[innerleftsep=1pt,innerrightsep=1pt,width=18pt,height=18pt]{\scriptsize{$F$}}{\scriptsize{$W$}} & $4$\\
\hline
$2$& $1$
\end{tabular}
\end{TAB}
\end{center}

For the cohomology of $\calX$, with Deligne's mixed Hodge structure, the filtrations are as follows. The filtrations on $H^{k}_{X_0}(\calX)$ may be obtained from these by Poincar\'{e} duality.

\begin{center}
\begin{TAB}{|c|c|c|c|c|c|c|}{|c|t|}
$\Gr_1^PH^0(\calX)$&&$\Gr_2^PH^2(\calX)$&$\Gr_3^PH^2(\calX)$&&$\Gr_4^PH^4(\calX)$ & $\Gr_5^PH^4(\calX)$\\
\begin{tabular}{c|c}
\diagbox[innerleftsep=1pt,innerrightsep=1pt,width=18pt,height=18pt]{\scriptsize{$F$}}{\scriptsize{$W$}} & $0$\\
\hline
$0$& $1$
\end{tabular}&&
\begin{tabular}{c|c}
\diagbox[innerleftsep=1pt,innerrightsep=1pt,width=18pt,height=18pt]{\scriptsize{$F$}}{\scriptsize{$W$}} & $2$\\
\hline
$1$& $g$
\end{tabular}&
\begin{tabular}{c|ccc}
\diagbox[innerleftsep=1pt,innerrightsep=1pt,width=18pt,height=18pt]{\scriptsize{$F$}}{\scriptsize{$W$}} & $0$ & $1$ & $2$\\
\hline
$0$& $1$ & $0$ &$0$\\
$1$ & $0$ & $0$ & $19$
\end{tabular}
&&
\begin{tabular}{c|c}
\diagbox[innerleftsep=1pt,innerrightsep=1pt,width=18pt,height=18pt]{\scriptsize{$F$}}{\scriptsize{$W$}} & $4$\\
\hline
$2$& $g$
\end{tabular}
&
\begin{tabular}{c|c}
\diagbox[innerleftsep=1pt,innerrightsep=1pt,width=18pt,height=18pt]{\scriptsize{$F$}}{\scriptsize{$W$}} & $4$\\
\hline
$2$& $1$
\end{tabular}
\end{TAB}
\end{center}

\subsection{Mirror symmetry for K3 surfaces}\label{sec:K3mirrorex}
Using this we can now prove Theorem \ref{thm:K3evidence}.

\subsubsection{Type II degenerations and genus $1$ fibrations} Suppose that $X$ is a generic $L$-polarized K3 surface, for some lattice $L$, that admits a Type II degeneration with $(r+1)$ components. Such a degeneration corresponds to a $1$-dimensional cusp in the Baily-Borel compactification of the moduli space of $L$-polarized K3 surfaces. Let $p$ be a $0$-dimensional cusp contained in this $1$-dimensional cusp; then we may apply the mirror construction of \cite{mslpk3s} at $p$ to construct a mirror K3 surface $Y$. Assuming $Y$ is generic, by \cite[Remark 4.2]{mstdfcym} $Y$ admits a fibration $\pi\colon Y \to \bP^1$ with general fibre a genus $1$ curve. Embed $\bP^1 \hookrightarrow \bP^r$ by the degree $r$ Veronese embedding and let $Z$ be the pullback of a generic hyperplane  in $\bP^r$, so that $Z$ is a union of $r$ disjoint fibres of $\pi$. Comparing the computations from Subsections \ref{sec:Bcurve} and \ref{sec:typeii}, it is then straightforward to verify that Theorem \ref{thm:K3evidence}(1) holds.

\subsubsection{Type III degenerations and finite morphisms} Next suppose that $Y$ is a generic K3 surface of degree $2k$. By definition, the polarizing line bundle $\calL$ on $Y$ is ample with $\calL^2  =2k$. A generic section of $\calL$ defines a  smooth curve $Z$ of genus $g = k + 1$ and $\calL$ defines a finite morphism $\pi\colon Y \to B \subset \bP^N$, where $B$ is a smooth surface \cite[Proposition 3]{fk3s}.

The mirror to $Y$, in the sense of \cite{mslpk3s}, is a K3 surface polarized by the lattice $M_k := H \oplus E_8 \oplus E_8 \oplus \langle -2k \rangle$, where $H$ denotes the hyperbolic plane lattice and $E_8$ is the negative definite $E_8$ lattice. By \cite[Proposition 1.16]{mfdnvfg2}, such K3 surfaces admit a unique (up to flops) primitive Type III degeneration, from which all other Type III degenerations may be obtained via base change. Moreover, the proof of this result shows that the central fibre of this Type III degeneration has $2k$ triple points.

Using the fact that every triple point lies at the intersection of $3$ double curves, along with the fact that the dual graph of the central fibre is a sphere which has topological Euler characteristic equal to $2$, one may check that a central fibre of a Type III degeneration with $2k$ triple points must  have $3k$ double curves and $k+2 = g+1$ components. We are therefore in the setting of Subsection \ref{sec:typeiii}. By comparing the computations from Subsections \ref{sec:Bsurface} and \ref{sec:typeiii}, it is then straightforward to verify that Theorem \ref{thm:K3evidence}(2) holds.

\begin{remark} It is natural to ask what happens if the conditions on $Y$ are relaxed to allow $Y$ to be a generic K3 surface with a lattice polarization of any rank. In this case we can choose $\calL$ to be any line bundle with $\calL^2 > 0$ that is generated by its global sections, to obtain a generically finite morphism $Y \to \bP^N$ with two-dimensional image \cite[Proposition 2]{fk3s}. Here the picture is much more complicated, but we can say more in a special case.

Suppose that $Y$ is polarized by the rank $19$ lattice $M_1 := H \oplus E_8 \oplus E_8 \oplus \langle -2 \rangle$. Choose a primitive integral vector $v$ in the nef cone of $Y$, with $v^2 = 2k > 0$, corresponding to a nef and big divisor of arithmetic genus $k+1$ on $Y$. The mirror to $Y$ is a K3 surface of degree $2$ and, by the results of \cite{spcmk3sd2}, the complex moduli space of K3 surfaces of degree $2$ admits a natural compactification with structure described by $\Nef(Y)$. In this compactification, the vector $v$ correponds to a Type III degeneration with $v^2 = 2k$ triple points and $k + 2$ components. This is in agreement with the prediction of Conjecture \ref{con:mirrorfiltrations}.\end{remark}

\subsubsection{Base change and Veronese embeddings} Given a semistable degeneration of K3 surfaces $ \calX \to \Delta$, one may perform a $\mu$-fold base change by the ramified covering $z \mapsto z^{\mu}$ and resolve singularities to obtain a new semistable degeneration. This new degeneration may be described geometrically by the method from \cite[Section 1]{bgd7}: briefly, the new degeneration has the same type as the original and, in addition to the original set of components of $X_0$, each double curve in $X_0$ gives rise to $(\mu-1)$ new components and each triple point gives $\frac{1}{2}(\mu-1)(\mu-2)$ new components.

\begin{enumerate}
\item If the original degeneration was of Type II with $r+1$ components in its central fibre, the $r$ double curves give rise to $(\mu-1)r$ new components (and there are no triple points), so the new degeneration has $\mu r + 1$ components.
\item If the original degeneration was of Type III with $k+2$ components in its central fibre, the $3k$ double curves give rise to $3(\mu-1)k$ new components  and the $2k$ triple points give rise to $(\mu-1)(\mu-2)k$ new components, so the new degeneration has $\mu^2 k + 2$ components.
\end{enumerate}

On the fibration side, we may re-embed $B \subset \bP^N \hookrightarrow \bP^M$, for $M = {\begin{psmallmatrix}N+1 \\ \mu \end{psmallmatrix} - 1}$, under the degree $\mu$ Veronese embedding, so that the divisor $Z$ is replaced by a smooth divisor in the linear system $|\mu Z|$. 

\begin{enumerate}
\item If $Z$ is a union of $r$ fibres in a genus one fibration, then a generic member of $|\mu Z|$ is a union of $\mu r$ fibres.
\item If $Z$ is a smooth ample curve of genus $g = k+1$, then a generic member of $|\mu Z|$ is a smooth ample curve of genus $g = \mu^2 k + 1$, by the genus formula.
\end{enumerate}
From these computations it is straightforward to verify Theorem \ref{thm:K3evidence}(3). 

\section{Calabi-Yau Threefolds}\label{sec:CY3}

In this final section we will prove Theorem \ref{thm:CY3evidence}, which provides some evidence for Conjecture \ref{con:mirrorfiltrations} in the Calabi-Yau threefold setting.

\subsection{K3 fibrations and Tyurin degenerations} \label{sec:tyurin} 

In the case of Calabi-Yau threefolds, \cite{mstdfcym} suggests a correspondence between Tyurin degenerations of a Calabi-Yau threefold and K3 fibrations on the mirror Calabi-Yau threefold; see Conjecture \ref{con:dht}. In this section we will show that Conjecture \ref{con:mirrorfiltrations} is a consequence of Conjecture \ref{con:dht}, thereby proving Theorem \ref{thm:CY3evidence}(1). We will then use this result to verify Conjecture \ref{con:mirrorfiltrations} explicitly in some examples where Conjecture \ref{con:dht} is known to hold. These examples will be important for the proof of Theorem \ref{thm:CY3evidence}(2) in Section \ref{sec:CY3elliptic}.

\subsubsection{K3 fibrations} We begin by assuming that $\pi\colon Y \to B$ is a fibration on a Calabi-Yau threefold whose general fibre is a K3 surface. Then necessarily we have $B \cong \bP^1$ and we take the trivial embedding of $\bP^1$ into itself, so that $C$ is a single point and $Z$ is a general K3 fibre in $Y$.

The filtrations in these cases will depend upon three parameters $u,v,w \in \bZ$, defined as follows.
\begin{itemize}
\item $u$ is the rank of the subgroup of $H^2(Z,\bZ)$ that is fixed by the action of monodromy around the singular fibres of $Y$.
\item If $Y_p$ denotes the fibre of $\pi \colon Y \to B$ over a point $p \in \bP^1$ and $\rho(Y_p)$ denotes the number of irreducible components of $Y_p$, then
\[v := \sum_{p \in \bP^1} (\rho(Y_p) - 1).\]
\item $w := h^{2,1}(Y)$ is the usual Hodge number.
\end{itemize}
Note that it follows from \cite[Lemma 3.2]{cytfmqk3s} that $h^{1,1}(Y) = u + v + 1$.

With this setup we may compute the following mixed Hodge structures $(W_{\bullet},F^{\bullet})$ on the graded pieces of the perverse Leray filtration $P_{\bullet}$. As usual, the entries in the tables show the dimensions of the graded pieces $\Gr_F^p\Gr^W_q\Gr^P_lH^k$. All other graded pieces of $P_{\bullet}$ are zero.

The cohomology of $Y$ can be computed easily using Proposition \ref{prop:cupperverse} and the explicit computation of the cup-product filtration in this setting from \cite[Section 5.1]{mstdfcym}. We obtain the following filtrations.
\begin{center}
\begin{TAB}{|c|c|c|c|c|c|c|}{|c|t|}
$\Gr_1^PH^0(Y)$&&$\Gr_1^PH^2(Y)$&$\Gr_2^PH^2(Y)$&$\Gr_3^PH^2(Y)$ && $\Gr_3^PH^3(Y)$\\
\begin{tabular}{c|c}
\diagbox[innerleftsep=1pt,innerrightsep=1pt,width=18pt,height=18pt]{\scriptsize{$F$}}{\scriptsize{$W$}} & $0$\\
\hline
$0$& $1$
\end{tabular}&&
\begin{tabular}{c|c}
\diagbox[innerleftsep=1pt,innerrightsep=1pt,width=18pt,height=18pt]{\scriptsize{$F$}}{\scriptsize{$W$}} & $2$\\
\hline
$1$& $1$
\end{tabular}&
\begin{tabular}{c|c}
\diagbox[innerleftsep=1pt,innerrightsep=1pt,width=18pt,height=18pt]{\scriptsize{$F$}}{\scriptsize{$W$}} & $2$\\
\hline
$1$& $v$ 
\end{tabular}&
\begin{tabular}{c|c}
\diagbox[innerleftsep=1pt,innerrightsep=1pt,width=18pt,height=18pt]{\scriptsize{$F$}}{\scriptsize{$W$}} & $2$\\
\hline
$1$& $u$
\end{tabular}
&&
\begin{tabular}{c|c}
\diagbox[innerleftsep=1pt,innerrightsep=1pt,width=18pt,height=18pt]{\scriptsize{$F$}}{\scriptsize{$W$}} & $3$\\
\hline
$0$& $1$\\
$1$ & $w$\\
$2$&$w$\\
$3$ & $1$
\end{tabular}
\end{TAB}
\smallskip

\begin{TAB}{|c|c|c|c|c|}{|c|t|}
$\Gr_3^PH^4(Y)$&$\Gr_4^PH^4(Y)$&$\Gr_5^PH^4(Y)$ && $\Gr_5^PH^6(Y)$\\
\begin{tabular}{c|c}
\diagbox[innerleftsep=1pt,innerrightsep=1pt,width=18pt,height=18pt]{\scriptsize{$F$}}{\scriptsize{$W$}} & $4$\\
\hline
$2$& $u$
\end{tabular}&
\begin{tabular}{c|c}
\diagbox[innerleftsep=1pt,innerrightsep=1pt,width=18pt,height=18pt]{\scriptsize{$F$}}{\scriptsize{$W$}} & $4$\\
\hline
$2$ & $v$
\end{tabular}&
\begin{tabular}{c|c}
\diagbox[innerleftsep=1pt,innerrightsep=1pt,width=18pt,height=18pt]{\scriptsize{$F$}}{\scriptsize{$W$}} & $4$\\
\hline
$2$& $1$
\end{tabular}
&&
\begin{tabular}{c|c}
\diagbox[innerleftsep=1pt,innerrightsep=1pt,width=18pt,height=18pt]{\scriptsize{$F$}}{\scriptsize{$W$}} & $6$\\
\hline
$3$& $1$
\end{tabular}
\end{TAB}
\end{center}

We obtain the following filtrations on the compactly supported cohomology of $U$ using the localization sequence (Proposition \ref{prop:perverselocalization}). The filtrations on the cohomology of $U$ can be obtained from these by Poincar\'{e}-Verdier duality.

\begin{center}
\begin{TAB}{|c|c|c|c|}{|c|t|}
$\Gr_1^PH^2_c(U)$&$\Gr_2^PH^2_c(U)$&&$\Gr_3^PH^3_c(U)$\\
\begin{tabular}{c|c}
\diagbox[innerleftsep=1pt,innerrightsep=1pt,width=18pt,height=18pt]{\scriptsize{$F$}}{\scriptsize{$W$}} & $2$\\
\hline
$1$& $1$
\end{tabular}&
\begin{tabular}{c|c}
\diagbox[innerleftsep=1pt,innerrightsep=1pt,width=18pt,height=18pt]{\scriptsize{$F$}}{\scriptsize{$W$}} & $2$\\
\hline
$1$& $v$
\end{tabular}&&
\begin{tabular}{c|cc}
\diagbox[innerleftsep=1pt,innerrightsep=1pt,width=18pt,height=18pt]{\scriptsize{$F$}}{\scriptsize{$W$}} & $2$&$3$\\
\hline
$0$& $1$ & $1$\\
$1$ & $20-u$ & $w$\\
$2$ & $1$ & $w$\\
$3$ & $0$ & $1$
\end{tabular}
\end{TAB}
\smallskip

\begin{TAB}{|c|c|c|c|}{|c|t|}
$\Gr_3^PH^4_c(U)$&$\Gr_4^PH^4_c(U)$&& $\Gr_5^PH^6_c(U)$\\
\begin{tabular}{c|c}
\diagbox[innerleftsep=1pt,innerrightsep=1pt,width=18pt,height=18pt]{\scriptsize{$F$}}{\scriptsize{$W$}} & $4$\\
\hline
$2$& $u$
\end{tabular}&
\begin{tabular}{c|c}
\diagbox[innerleftsep=1pt,innerrightsep=1pt,width=18pt,height=18pt]{\scriptsize{$F$}}{\scriptsize{$W$}} & $4$\\
\hline
$2$& $v$
\end{tabular}&&
\begin{tabular}{c|c}
\diagbox[innerleftsep=1pt,innerrightsep=1pt,width=18pt,height=18pt]{\scriptsize{$F$}}{\scriptsize{$W$}} & $6$\\
\hline
$3$& $1$
\end{tabular}
\end{TAB}
\end{center}

\subsubsection{Tyurin degenerations}

Now let $\calX \to \Delta$ be a Tyurin degeneration of Calabi-Yau threefolds (see Definition \ref{def:tyurin}). Recall that such degenerations have length $2$. The central fibre $X_0$ is a union of two quasi-Fano threefolds $V_1 \cup_S V_2$ glued along an anticanonical K3 surface $S$. As in the previous subsection, the filtrations in these cases will depend upon three parameters $\check{u},\check{v},\check{w} \in \bZ$, defined as follows.
\begin{itemize}
\item If $r_i\colon H^2(V_i,\bQ) \to H^2(S,\bQ)$ denotes the restriction, for $i = 1,2$, then $\check{u} := 20 - \rank(\im(r_1) + \im(r_2))$. 
\item $\check{v} := h^{2,1}(V_1) + h^{2,1}(V_2)$.
\item $\check{w} := h^{1,1}(X)$.
\end{itemize}
In this case it follows from \cite[Corollary 8.2]{cycsncv} that $h^{2,1}(X) = \check{u} + \check{v} + 1$.

\begin{remark} Geometrically, $\rank(\im(r_1) + \im(r_2))$ is the rank of the lattice polarization on $S$ induced by restricting divisors from $V_1$ and $V_2$, and $\check{u}$ is the rank of its mirror lattice in the sense of \cite{mslpk3s}.
\end{remark}

With this setup, we describe below the mixed Hodge structures $(W_{\bullet},F^{\bullet})$ on the graded pieces of the filtration $P_{\bullet}$. The entries in the tables show the dimensions of the graded pieces $\Gr_F^p\Gr^W_q\Gr^P_lH^k$. All other graded pieces of $P_{\bullet}$ are zero.

The limiting mixed Hodge structure on the cohomology of $X$ is computed in \cite[Section 5.1]{mstdfcym}; the filtrations obtained are as follows.

\begin{center}
\begin{TAB}{|c|c|c|c|c|c|c|c|c|}{|c|t|}
$\Gr_0^PH^0_{\lim}(X)$&&$\Gr_2^PH^2_{\lim}(X)$ &&$\Gr_3^PH^3_{\lim}(X)$&& $\Gr_4^PH^4_{\lim}(X)$&& $\Gr_6^PH^6_{\lim}(X)$\\
\begin{tabular}{c|c}
\diagbox[innerleftsep=1pt,innerrightsep=1pt,width=18pt,height=18pt]{\scriptsize{$F$}}{\scriptsize{$W$}} & $0$\\
\hline
$0$& $1$
\end{tabular}
&&
\begin{tabular}{c|c}
\diagbox[innerleftsep=1pt,innerrightsep=1pt,width=18pt,height=18pt]{\scriptsize{$F$}}{\scriptsize{$W$}} & $2$\\
\hline
$1$& $\check{w}$
\end{tabular}
&&
\begin{tabular}{c|ccc}
\diagbox[innerleftsep=1pt,innerrightsep=1pt,width=18pt,height=18pt]{\scriptsize{$F$}}{\scriptsize{$W$}} & $2$ & $3$ & $4$\\
\hline
$0$& $1$ & $0$ & $0$ \\
$1$ & $\check{u}$ & $\check{v}$ & $1$ \\
$2$ & $1$ & $\check{v}$ & $\check{u}$\\
$3$ & $0$ & $0$ & $1$
\end{tabular}
&&
\begin{tabular}{c|c}
\diagbox[innerleftsep=1pt,innerrightsep=1pt,width=18pt,height=18pt]{\scriptsize{$F$}}{\scriptsize{$W$}} & $4$\\
\hline
$2$& $\check{w}$
\end{tabular}
&&
\begin{tabular}{c|c}
\diagbox[innerleftsep=1pt,innerrightsep=1pt,width=18pt,height=18pt]{\scriptsize{$F$}}{\scriptsize{$W$}} & $6$\\
\hline
$3$& $1$
\end{tabular}
\end{TAB}
\end{center}

For the cohomology of $\calX$, with Deligne's mixed Hodge structure, the filtrations are easily computed using the standard spectral sequence for a normal crossing degeneration. The filtrations on $H^{k}_{X_0}(\calX)$ may be obtained from these by Poincar\'{e} duality.

\begin{center}
\begin{TAB}{|c|c|c|c|c|c|}{|c|t|}
$\Gr_1^PH^0(\calX)$&&$\Gr_2^PH^2(\calX)$&$\Gr_3^PH^2(\calX)$&&$\Gr_4^PH^3(\calX)$\\
\begin{tabular}{c|c}
\diagbox[innerleftsep=1pt,innerrightsep=1pt,width=18pt,height=18pt]{\scriptsize{$F$}}{\scriptsize{$W$}} & $0$\\
\hline
$0$& $1$
\end{tabular}&&
\begin{tabular}{c|c}
\diagbox[innerleftsep=1pt,innerrightsep=1pt,width=18pt,height=18pt]{\scriptsize{$F$}}{\scriptsize{$W$}} & $2$\\
\hline
$1$& $1$
\end{tabular}&
\begin{tabular}{c|c}
\diagbox[innerleftsep=1pt,innerrightsep=1pt,width=18pt,height=18pt]{\scriptsize{$F$}}{\scriptsize{$W$}} & $2$\\
\hline
$1$& $\check{w}$
\end{tabular}&&
\begin{tabular}{c|cc}
\diagbox[innerleftsep=1pt,innerrightsep=1pt,width=18pt,height=18pt]{\scriptsize{$F$}}{\scriptsize{$W$}} & $2$ & $3$\\
\hline
$0$& $1$ & $0$\\
$1$ & $\check{u}$ & $\check{v}$\\
$2$ & $1$ & $\check{v}$
\end{tabular}
\end{TAB}
\smallskip

\begin{TAB}{|c|c|c|c|c|}{|c|t|}
$\Gr_4^PH^4(\calX)$&$\Gr_5^PH^4(\calX)$ && $\Gr_6^PH^6(\calX)$& $\Gr_7^PH^6(\calX)$\\
\begin{tabular}{c|c}
\diagbox[innerleftsep=1pt,innerrightsep=1pt,width=18pt,height=18pt]{\scriptsize{$F$}}{\scriptsize{$W$}} & $4$\\
\hline
$2$& $20 - \check{u}$
\end{tabular}
&
\begin{tabular}{c|c}
\diagbox[innerleftsep=1pt,innerrightsep=1pt,width=18pt,height=18pt]{\scriptsize{$F$}}{\scriptsize{$W$}} & $4$\\
\hline
$2$& $\check{w}$
\end{tabular}
&&
\begin{tabular}{c|c}
\diagbox[innerleftsep=1pt,innerrightsep=1pt,width=18pt,height=18pt]{\scriptsize{$F$}}{\scriptsize{$W$}} & $6$\\
\hline
$3$& $1$
\end{tabular}
&
\begin{tabular}{c|c}
\diagbox[innerleftsep=1pt,innerrightsep=1pt,width=18pt,height=18pt]{\scriptsize{$F$}}{\scriptsize{$W$}} & $6$\\
\hline
$3$& $1$
\end{tabular}
\end{TAB}
\end{center}

\subsubsection{Mirror symmetry} From the computations above, to verify Conjecture \ref{con:mirrorfiltrations} in this setting it suffices to demonstrate equality between the parameters $u,v,w$ for $Y$ and $\check{u},\check{v}, \check{w}$ for $X$.

Assume that Conjecture \ref{con:dht} holds for $\calX \to \Delta$ and $\pi\colon Y \to \bP^1$. Then the equalities $w = \check{w}$ and $u + v = \check{u} + \check{v}$ are just the usual mirror symmetry of Hodge numbers for $X$ and $Y$. Moreover, equality of $u$ and $\check{u}$ follows from the fact that $S$ and $Z$ are mirror in the sense of \cite{mslpk3s}. This proves Theorem \ref{thm:CY3evidence}(1). 

%Some progress towards establishing compatibility between these conjectures and the Batyrev-Borisov approach to mirror symmetry \cite{tmscycigtfv,dpmscyhtv,cycitv} has been made in \cite[Section 3]{mstdfcym}, but the equality of $u$ and $\check{u}$ depends subtly on the existence of reducible singular members of certain pencils of divisors on a toric Calabi-Yau threefold, which has only been verified in certain cases.

\subsubsection{An explicit example} Theorem \ref{thm:CY3evidence}(1) allows us to verify Conjecture \ref{con:mirrorfiltrations} in settings where Conjecture \ref{con:dht} is known to hold. A good source of such examples is provided by degenerations of Calabi-Yau threefolds to pairs of Fano threefolds of Picard rank one glued along a smooth anticanonical K3 surface $S$. In this case one must have $\check{u} = 1$, so $S$ is a K3 surface polarized by the rank one lattice $\langle 2n \rangle$ for some $n > 0$, and one can look for a fibration of the mirror Calabi-Yau threefold by K3 surfaces polarized by the mirror lattice $M_n := H \oplus E_8 \oplus E_8 \oplus \langle -2n \rangle$. This is made simpler by the fact that such K3-fibred Calabi-Yau threefolds have been completely classified in the papers \cite{cytfhrlpk3s} (for $n \geq 2$) and \cite{tfmsdp} (for $n = 1$). 

The explicit construction of such examples closely follows the approach from \cite[Section 5]{mstdfcym}. We give an illustrative example, which will be important for the proof of Theorem \ref{thm:CY3evidence}(2).

\begin{example}\label{ex:tyurinquintic}
Let $X_5$ be the quintic threefold in $\bP^4$. We consider the following two degenerations of $X_5$.
\begin{enumerate}
\item $X_5$ degenerates to the union of a hyperplane $V_1$ and a quartic surface $V_4$. To construct such a degeneration, consider the family of threefolds defined by $\{lq_4 =tf_5\} \subset \bP^4 \times \Delta$, where $l$, $q_4$, and $f_5$ are generic homogeneous equations of degrees $1$, $4$, and $5$ in the coordinates of $\bP^4$, and $t$ is a parameter on the complex disc $\Delta$. The central fibre $t = 0$ consists of a hyperplane and a quartic threefold meeting along $\{q_4 = l = 0\} \subset \bP^4$, which is a smooth K3 surface. Restriction of a hyperplane section induces a lattice polarization on this K3 surface by the rank one lattice $\langle 4 \rangle$. 

The fourfold $\{lq_4 =tf_5\}$ is singular along the curve $C = \{l = q_4 = t = f_5 = 0\}$, which has genus $51$. This curve of singularities may be resolved by blowing up $\{t = l = 0\}$ in the ambient space $\bP^4 \times \Delta$ and taking the strict transform; the result is a Tyurin degeneration of $X_5$. For this degeneration the parameters $\check{u},\check{v},\check{w}$ can be computed to be
\begin{align*}
\check{u} &= 20 - \rank(\langle 4 \rangle) = 20 -1 = 19, \\
\check{v} &= h^{2,1}(V_1) + h^{2,1}(V_4) + g(C) = 0 + 30 + 51 = 81,\\
\check{w} &= h^{1,1}(X_5) = 1,
\end{align*}
where the formula for $\check{v}$ follows from standard formulas for the cohomology of a blow-up (see, for example, \cite[Theorem 7.31]{htcagi}.
\item $X_5$ degenerates to the union of a quadric $V_2$ and a cubic surface $V_3$. To construct such a degeneration, consider the family of threefolds defined by $\{q_2q_3 =tf_5\} \subset \bP^4 \times \Delta$, where $q_2$, $q_3$, and $f_5$ are generic homogeneous equations of degrees $2$, $3$, and $5$ in the coordinates of $\bP^4$, and $t$ is a parameter on the complex disc $\Delta$. The central fibre $t = 0$ consists of a quadric and a cubic threefold meeting along $\{q_2 = q_3 = 0\} \subset \bP^4$, which is a smooth K3 surface. Restriction of a hyperplane section induces a lattice polarization on this K3 surface by the rank one lattice $\langle 6 \rangle$. 

The fourfold $\{q_2q_3 =tf_5\}$ is singular along the curve $C' = \{q_2 = q_3 = t = f_5 = 0\}$, which has genus $76$. This curve of singularities may be resolved by blowing up $\{t = q_2 = 0\}$ in the ambient space $\bP^4 \times \Delta$ and taking the strict transform; the result is a Tyurin degeneration of $X_5$. For this degeneration the parameters $\check{u},\check{v},\check{w}$ can be computed to be
\begin{align*}
\check{u} &= 20 - \rank(\langle 6 \rangle) = 20 -1 = 19, \\
\check{v} &= h^{2,1}(V_2) + h^{2,1}(V_3) + g(C') = 0 + 5 + 76 = 81,\\
\check{w} &= h^{1,1}(X_5) = 1.
\end{align*}
\end{enumerate}

Now we proceed to the quintic mirror threefold $Y$. It was shown in \cite[Theorem 5.10]{flpk3sm} that this threefold admits K3 fibrations $Y \to \bP^1$ polarized by the rank $19$ lattices $M_2$ and $M_3$. For the $M_2$-polarized fibration, the parameter $v$ has been computed before in \cite[Example 3.11]{cytfmqk3s}. For the $M_3$-polarized fibration, it follows easily from the explicit descriptions of the singular fibres of such fibrations from \cite[Section 3.2]{cytfhrlpk3s}.
\begin{enumerate}
\item The $M_2$-polarized fibration has two reducible singular fibres with $31$ and $52$ components respectively. We thus obtain
\begin{align*}
{u} &= \rank(M_2) = 19, \\
{v} &= \sum_{p \in \bP^1} (\rho(Y_p) - 1) = 30 + 51 = 81\\
{w} &= h^{2,1}(X_5) = 1.
\end{align*}
\item The $M_3$-polarized fibration has two reducible singular fibres with $6$ and $77$ components respectively. We thus obtain
\begin{align*}
{u} &= \rank(M_3) = 19, \\
{v} &= \sum_{p \in \bP^1} (\rho(Y_p) - 1) = 5 + 76 = 81\\
{w} &= h^{2,1}(X_5) = 1.
\end{align*}
\end{enumerate}

We have thus verified Conjecture \ref{con:mirrorfiltrations} explicitly in this setting.
\end{example}

\begin{remark} In light of the conjectures of \cite{mstdfcym}, it is no coincidence that the numbers ($30$, $51$, $5$, $76$) showing up in the decompositions of $v$ and $\check{v}$ above are identical. Indeed, these are conjectured to represent invariants of certain quasi-Fano/LG mirror pairs.
\end{remark}

\subsection{Elliptic fibrations and degenerations of length 3}\label{sec:CY3elliptic}

The picture for elliptic fibrations and degenerations of length $3$ of Calabi-Yau threefolds is substantially more complicated and general results are beyond our existing methods. There are two main difficulties. Firstly, a given Calabi-Yau may admit many distinct degenerations of length $3$ and its mirror may have many elliptic fibrations, and it is not easy to see how these should be matched together. Secondly, it is not straightforward to compute the filtrations on the elliptic fibration side: the perverse filtration, in particular, requires detailed geometric knowledge of the singular fibre structure of the elliptic fibration, which is not readily available in many cases.

In this section we will present an explicit example in this setting, involving a degeneration of length $3$ of the quintic threefold and an elliptic fibration on its mirror. A mirror relationship between this degeneration and fibration was first observed by Doran, Kostiuk, and You \cite[Section 7]{dfhrlgm}, and the geometry of the elliptic fibration has been well-studied in \cite{cymrsrmt} and \cite{dfhrlgm}. This will allow us to verify Conjecture \ref{con:mirrorfiltrations} in this setting and thereby prove Theorem \ref{thm:CY3evidence}(2)

\subsubsection{A degeneration of length 3 of the quintic threefold}\label{sec:type3CY3}

We begin by constructing a degeneration of length $3$ of the quintic threefold $X_5 \subset \bP^4$ and computing its filtrations. This degeneration is a common degeneration of the two Tyurin degenerations from Example \ref{ex:tyurinquintic}; this will be important when we come to study its mirror.

Consider the degeneration of $X_5$ to a union of two hyperplanes $V_1$ and $V_1'$ and a cubic surface $V_3$. To construct such a degeneration, consider the family of threefolds defined by $\{ll'q_3 =tf_5\} \subset \bP^4 \times \Delta$, where $l$, $l'$, $q_3$, and $f_5$ are generic homogeneous equations of degrees $1$, $1$, $3$, and $5$ in the coordinates of $\bP^4$, and $t$ is a parameter on the complex disc $\Delta$. The central fibre $t = 0$ consists of two hyperplanes $V_1 := \{l = 0\}$ and $V_1' := \{l' = 0\}$ and a cubic threefold $V_3:=\{q_3 = 0\}$. The double intersection $V_1 \cap V_1'$ is a copy of $\bP^2$, whilst $V_1 \cap V_3$ and $V_1' \cap V_3$ are both isomorphic to cubic surfaces in $\bP^3$. Finally, the triple intersection $V_1 \cap V_1' \cap V_3$ is isomorphic to a cubic curve in $\bP^2$, which has genus $1$.

The fourfold $\{ll'q_3 =tf_5\}$ is singular along the curves $C_1 = \{l = l' = t = f_5 = 0\}$, which has genus $6$, and $C_2 = \{l = q_3 = t = f_5 = 0\}$ and $C_3 = \{l' = q_3 = t = f_5 = 0\}$, which both have genus $31$. We resolve these curves of singularities by blowing up $\{t = q_3 = 0\}$ followed by $\{t = l = 0\}$ in the ambient space $\bP^4 \times \Delta$ and taking the strict transform; the result is a $K$-trivial semistable degeneration $\calX \to \Delta$, with general fibre a quintic threefold $X_5$. 

The central fibre of this degeneration consists of three components $\hat{V}_1$, $\hat{V}_1'$ and $\hat{V}_3$, which are the strict transforms of $V_1$, $V_1'$, and $V_3$ respectively. $\hat{V}_1$ is isomorphic to $\bP^3$ blown up in a curve of genus $6$, $\hat{V}_1'$ is isomorphic to $\bP^3$, and $\hat{V}_3$ is isomorphic to a cubic threefold blown up in two curves of genus $31$. The double locus $\hat{V}_1 \cap \hat{V}_1'$ is isomorphic to $\bP^2$, the double locus $\hat{V}_1' \cap \hat{V}_3$ is isomorphic to a cubic surface, and $\hat{V}_1 \cap \hat{V}_3$ is isomorphic to a cubic surface that has been blown up in $15$ points on a smooth anticanonical curve. The triple locus is a smooth curve of genus $1$.

If we let $V^{[i]}$ denote the disjoint union of the codimension $i$ components of the central fibre, then the difficult parts in the spectral sequence for a normal crossing degeneration can be computed to be:
\begin{align*}
h^2(V^{[0]}) &= h^2(\hat{V_1}) + h^2(\hat{V}_1') + h^2(\hat{V}_3) = 2 + 1 + 3 = 6,\\
h^3(V^{[0]}) &= h^3(\hat{V_1}) + h^3(\hat{V}_1') + h^3(\hat{V}_3) =  (0 + 2.6) + 0 + (10 + 2.31 + 2.31) = 146,\\
h^2(V^{[1]}) &= h^2(\hat{V}_1 \cap \hat{V}_1') + h^2(\hat{V}_1' \cap \hat{V}_3) + h^2(\hat{V}_1 \cap \hat{V}_3) = 1 + 7 + (7+15) = 30.
\end{align*}
Computation in this spectral sequence yields the weight filtration in Deligne's mixed Hodge structure and, combined with some diagram chasing in the Clemens-Schmid sequence, we obtain the following mixed Hodge structures $(W_{\bullet},F^{\bullet})$ on the graded pieces of the filtration $P_{\bullet}$. The entries in the tables show the dimensions of the graded pieces $\Gr_F^p\Gr^W_q\Gr^P_lH^k$. All other graded pieces of $P_{\bullet}$ are zero.

For the cohomology of $X_5$, with its limiting mixed Hodge structure, the filtrations are as follows. It follows from this description that this degeneration has length $3$.
\begin{center}
\begin{TAB}{|c|c|c|c|c|}{|c|t|}
$\Gr_0^PH^0_{\lim}(X)$&&$\Gr_2^PH^2_{\lim}(X)$ &&$\Gr_3^PH^3_{\lim}(X)$\\
\begin{tabular}{c|c}
\diagbox[innerleftsep=1pt,innerrightsep=1pt,width=18pt,height=18pt]{\scriptsize{$F$}}{\scriptsize{$W$}} & $0$\\
\hline
$0$& $1$
\end{tabular}
&&
\begin{tabular}{c|c}
\diagbox[innerleftsep=1pt,innerrightsep=1pt,width=18pt,height=18pt]{\scriptsize{$F$}}{\scriptsize{$W$}} & $2$\\
\hline
$1$& $1$
\end{tabular}
&&
\begin{tabular}{c|ccccc}
\diagbox[innerleftsep=1pt,innerrightsep=1pt,width=18pt,height=18pt]{\scriptsize{$F$}}{\scriptsize{$W$}} & $1$ & $2$ & $3$ & $4$ & $5$\\
\hline
$0$& $1$ & $0$ & $0$ & $0$ & $0$ \\
$1$ & $1$ & $26$ & $74$ & $0$ & $0$ \\
$2$ & $0$ & $0$ &$74$ & $26$ & $1$\\
$3$ & $0$ & $0$ & $0$ & $0$ & $1$
\end{tabular}
\end{TAB}
\smallskip

\begin{TAB}{|c|c|c|}{|c|t|}
 $\Gr_4^PH^4_{\lim}(X)$&& $\Gr_6^PH^6_{\lim}(X)$\\
\begin{tabular}{c|c}
\diagbox[innerleftsep=1pt,innerrightsep=1pt,width=18pt,height=18pt]{\scriptsize{$F$}}{\scriptsize{$W$}} & $4$\\
\hline
$2$& $1$
\end{tabular}
&&
\begin{tabular}{c|c}
\diagbox[innerleftsep=1pt,innerrightsep=1pt,width=18pt,height=18pt]{\scriptsize{$F$}}{\scriptsize{$W$}} & $6$\\
\hline
$3$& $1$
\end{tabular}
\end{TAB}
\end{center}

For the cohomology of $\calX$, with Deligne's mixed Hodge structure, the filtrations are as follows. The filtrations on $H^{k}_{X_0}(\calX)$ may be obtained from these by Poincar\'{e} duality.

\begin{center}
\begin{TAB}{|c|c|c|c|c|c|}{|c|t|}
$\Gr_1^PH^0(\calX)$&&$\Gr_2^PH^2(\calX)$&$\Gr_3^PH^2(\calX)$&&$\Gr_4^PH^3(\calX)$\\
\begin{tabular}{c|c}
\diagbox[innerleftsep=1pt,innerrightsep=1pt,width=18pt,height=18pt]{\scriptsize{$F$}}{\scriptsize{$W$}} & $0$\\
\hline
$0$& $1$
\end{tabular}&&
\begin{tabular}{c|c}
\diagbox[innerleftsep=1pt,innerrightsep=1pt,width=18pt,height=18pt]{\scriptsize{$F$}}{\scriptsize{$W$}} & $2$\\
\hline
$1$& $2$
\end{tabular}&
\begin{tabular}{c|c}
\diagbox[innerleftsep=1pt,innerrightsep=1pt,width=18pt,height=18pt]{\scriptsize{$F$}}{\scriptsize{$W$}} & $2$\\
\hline
$1$& $1$
\end{tabular}&&
\begin{tabular}{c|ccc}
\diagbox[innerleftsep=1pt,innerrightsep=1pt,width=18pt,height=18pt]{\scriptsize{$F$}}{\scriptsize{$W$}} & $1$ & $2$ & $3$\\
\hline
$0$& $1$ & $0$ & $0$\\
$1$ & $1$& $26$ & $73$\\
$2$ & $0$ & $0$ & $73$
\end{tabular}
\end{TAB}
\smallskip

\begin{TAB}{|c|c|c|c|c|}{|c|t|}
$\Gr_4^PH^4(\calX)$&$\Gr_5^PH^4(\calX)$ && $\Gr_6^PH^6(\calX)$& $\Gr_7^PH^6(\calX)$\\
\begin{tabular}{c|c}
\diagbox[innerleftsep=1pt,innerrightsep=1pt,width=18pt,height=18pt]{\scriptsize{$F$}}{\scriptsize{$W$}} & $4$\\
\hline
$2$& $2$
\end{tabular}
&
\begin{tabular}{c|c}
\diagbox[innerleftsep=1pt,innerrightsep=1pt,width=18pt,height=18pt]{\scriptsize{$F$}}{\scriptsize{$W$}} & $4$\\
\hline
$2$& $1$
\end{tabular}
&&
\begin{tabular}{c|c}
\diagbox[innerleftsep=1pt,innerrightsep=1pt,width=18pt,height=18pt]{\scriptsize{$F$}}{\scriptsize{$W$}} & $6$\\
\hline
$3$& $2$
\end{tabular}
&
\begin{tabular}{c|c}
\diagbox[innerleftsep=1pt,innerrightsep=1pt,width=18pt,height=18pt]{\scriptsize{$F$}}{\scriptsize{$W$}} & $6$\\
\hline
$3$& $1$
\end{tabular}
\end{TAB}
\end{center}

\subsubsection{An elliptic fibration on the mirror quintic threefold}

Conjecture \ref{con:mirrorfiltrations} predicts that the mirror quintic threefold $Y$ admits an elliptic fibration whose Hodge, weight, and perverse Leray filtrations mirror those for the degeneration of length $3$ above. In this subsection we describe such a fibration and show that it has the required properties.

The degeneration of length $3$ above occurred as a common degeneration of the two Tyurin degenerations studied in Example \ref{ex:tyurinquintic}. It seems reasonable to expect, therefore, that the mirror elliptic fibration enjoys some compatibility with the $M_2$- and $M_3$-polarized K3 surface fibrations studied in that example.

An elliptic fibration with the desired properties has been studied by Doran and Malmendier \cite[Section 7.3]{cymrsrmt} and Doran, Kostiuk, and You \cite[Section 7]{dfhrlgm}. However, these two descriptions are not immediately compatible: each describes the fibration on an affine chart in the base of the fibration and the chosen affine charts are quite different. We opt to use the description from \cite[Section 7]{dfhrlgm}, as it makes both of the K3 fibrations readily apparent.

This elliptic fibration is described as follows. There is a family of elliptic curves with section over $\bP^1$, with base parameter $\lambda$, with singular fibres of Kodaira types $\mathrm{I}_3$ at $\lambda = 0$, $\mathrm{IV}^*$ at $\lambda = \infty$, and $\mathrm{I}_1$ at $\lambda = \frac{1}{3^3}$; \cite{dfhrlgm} call this the \emph{mirror cubic family}. In \cite[Section 7]{dfhrlgm}, Doran, Kostiuk, and You show that there is an open set in the quintic mirror threefold $Y$ on which the elliptic fibration is isomorphic to the family of elliptic curves defined by the pull-back of the mirror cubic family under the map $\bA^2 \to \bP^1$ given by
\[ \lambda = \frac{k}{x(1-y)(y-x)^3},\]
for some constant $k$.

In this description, lines $\{x = c\} \subset \bA^2$ (resp. $\{y = c\} \subset \bA^2$), for $c$ constant, define elliptic surfaces in $Y$ with section and singular fibres of Kodaira type $\mathrm{IV}^*$ at $y=1$ (resp. $x = 0$), type $\mathrm{I}_{12}$ at infinity, plus an additional $4 \times \mathrm{I}_1$. Such surfaces are ramified $4$-fold covers of the mirror cubic family. Varying $c$ gives two pencils of $M_2$-polarized K3 surfaces, defining $M_2$-polarized K3 fibrations on $Y$ (see \cite[Remark 7.2]{dfhrlgm}; compare also \cite[Lemma 7.3]{cymrsrmt}).

Moreover, lines $\{x - y = c\} \subset \bA^2$, for $c$ constant, define elliptic surfaces in $Y$ with section and singular fibres of Kodaira type $2 \times \mathrm{IV}^*$ at the intersection with the lines $\{x = 0\}$ and $\{y = 1\}$, type $\mathrm{I}_6$ at infinity, plus an additional $2 \times \mathrm{I}_1$. Such surfaces are ramified double covers of the mirror cubic family. Varying $c$ gives a pencil of $M_3$-polarized K3 surfaces, defining an $M_3$-polarized K3 fibration on $Y$ (see \cite[Remark 7.2]{dfhrlgm}).

Now we set up our fibration, as in Section \ref{sec:notation}. Define a base point free linear system on $Y$ by taking the closure of preimages of general lines from $\bA^2$. This gives a map $\pi\colon Y \to B = \bP^2$, whose general fibre is an elliptic curve (note that the special fibres may be quite unpleasant, even involving components of dimension $2$, but this is not forbidden in our construction). We embed $\bP^2$ into itself by the trivial embedding, so that $C \subset \bP^2$ is a general line, and we let $Z = \pi^{-1}(C) \subset Y$. It follows from the description above that $Z$ is an elliptic surface with section and singular fibres of Kodaira types $2\times \mathrm{IV}^*$ where $C$ intersects the lines $\{x = 0\}$ and $\{y = 1\}$, type $\mathrm{I}_{15}$ at infinity, plus an additional $5 \times \mathrm{I}_1$. This surface has topological Euler characteristic $36$ and is a ramified $5$-fold cover of the mirror cubic family.

With this information, using Proposition \ref{prop:cupperverse} it is straightforward to compute the mixed Hodge structures $(W_{\bullet},F^{\bullet})$ on the graded pieces of the perverse filtration $P_{\bullet}$ for the cohomology $H^k(Z)$. We obtain the following filtrations.

\begin{center}
\begin{TAB}{|c|c|c|c|c|c|c|}{|c|t|}
$\Gr_1^PH^0(Z)$&&$\Gr_1^PH^2(Z)$&$\Gr_2^PH^2(Z)$&$\Gr_3^PH^2(Z)$&& $\Gr_3^PH^4(Z)$\\
\begin{tabular}{c|c}
\diagbox[innerleftsep=1pt,innerrightsep=1pt,width=18pt,height=18pt]{\scriptsize{$F$}}{\scriptsize{$W$}} & $0$\\
\hline
$0$& $1$
\end{tabular}&&
\begin{tabular}{c|c}
\diagbox[innerleftsep=1pt,innerrightsep=1pt,width=18pt,height=18pt]{\scriptsize{$F$}}{\scriptsize{$W$}} & $2$\\
\hline
$1$&$1$
\end{tabular}&
\begin{tabular}{c|c}
\diagbox[innerleftsep=1pt,innerrightsep=1pt,width=18pt,height=18pt]{\scriptsize{$F$}}{\scriptsize{$W$}} & $2$\\
\hline
$0$& $2$\\
$1$&$28$\\
$2$&$2$
\end{tabular}&
\begin{tabular}{c|c}
\diagbox[innerleftsep=1pt,innerrightsep=1pt,width=18pt,height=18pt]{\scriptsize{$F$}}{\scriptsize{$W$}} & $2$\\
\hline
$1$& $1$
\end{tabular}
&&
\begin{tabular}{c|c}
\diagbox[innerleftsep=1pt,innerrightsep=1pt,width=18pt,height=18pt]{\scriptsize{$F$}}{\scriptsize{$W$}} & $4$\\
\hline
$2$& $1$
\end{tabular}
\end{TAB}
\end{center}

Proposition \ref{prop:cupperverse}, applied to $Y$, together with a diagram chase in the localization exact sequence (Proposition \ref{prop:perverselocalization}) is sufficient to compute most of the structure of the filtrations on the cohomology of $Y$ and $U$ where, as usual, $U := Y - Z$. To complete the computation, it suffices to find the image of the pull-back $i^*\colon H^{1,1}(Y) \to H^{1,1}(Z)$ appearing in the localization sequence, where $i\colon Z \hookrightarrow Y$ denotes the inclusion.

To compute this image, we use the fact that the structure of the $M_2$-fibration on $Y$ is completely understood, by the results of \cite{cytfmqk3s} (see also \cite{cytfhrlpk3s} for a generalization of these results). In particular, \cite[Example 3.11]{cytfmqk3s} gives the following set of divisors which generate $H^{1,1}(Y)$.
\begin{enumerate}
\item A general K3 fibre in the $M_2$-polarized fibration, corresponding to a line $\{x = c\} \subset \bA^2$, for general $c$.
\item $30$ divisors are irreducible components of a singular fibre in the $M_2$-polarized fibration, corresponding to the line $\{x = 0\} \subset \bA^2$.
\item $51$ divisors are irreducible components of a semistable singular fibre in the $M_2$-polarized fibration, corresponding to the line at infinity.
\item $19$ monodromy invariant divisors arising from the $M_2$-polarization on a general fibre. Over $\bQ$, the space spanned by these divisors is generated by irreducible components of the $\mathrm{IV}^*$ and $\mathrm{I}_{12}$ singular fibres in the elliptic fibration on an $M_2$-polarized K3 surface, plus the classes of a section and a fibre; this corresponds to the finite index embedding of the lattice $H \oplus E_6 \oplus A_{11}$ into $M_2$ (see, for example, \cite[Section 6.3]{cmsak3s}).
\end{enumerate}
Now we examine how these divisors restrict to $Z$.
\begin{enumerate}
\item The intersection of $C$ with $\{x = c\}$, for general $c$, gives a smooth elliptic fibre in $Z$. So the class in $H^{1,1}(Y)$ of a general K3 fibre in the $M_2$-polarized fibration restricts to the class of an elliptic fibre in $H^{1,1}(Z)$.
\item The intersection of $C$ with $\{x=0\}$ gives a singular fibre of type $\mathrm{IV}^*$ in $Z$ and the $30$ divisors over $\{x = 0\}$ in $H^{1,1}(Y)$ restrict to the $6$ classes in $H^{1,1}(Z)$ coming from the first $\mathrm{IV}^*$ fibre.
\item The intersection of $C$ with the line at infinity gives a singular fibre of type $\mathrm{I}_{15}$ in $Z$ and the $51$ divisors over the line at infinity in $H^{1,1}(Y)$ restrict to the $14$ classes in $H^{1,1}(Z)$ coming from the $\mathrm{I}_{15}$ fibre.
\item The $6$ monodromy invariant divisors in $H^{1,1}(Y)$ coming  from components of the $\mathrm{IV}^*$ fibre in an $M_2$-polarized K3 surface lie over the line  $\{y=1\}$. The intersection of $C$ with $\{y=1\}$ gives the second  $\mathrm{IV}^*$ fibre in $Z$, so these $6$ divisors restrict to the $6$ classes in $H^{1,1}(Z)$ coming from this  $\mathrm{IV}^*$ fibre. The $12$ monodromy invariant divisors in $H^{1,1}(Y)$ coming  from components of the $\mathrm{I}_{12}$ fibre in an $M_2$-polarized K3 surface all lie over the base point at infinity of the linear system $\{x = c\}$, and $C$ does not pass through this point, so these restrict trivially to $H^{1,1}(Z)$. Finally, the $2$ monodromy invariant divisors in $H^{1,1}(Y)$ coming from a section and fibre in an $M_2$-polarized K3 surface restrict to a section and fibre of the elliptic fibration on $Z$; of these two we already accounted for the fibre in (1), above. Combining everything, we thus see that the $19$ monodromy invariant divisors arising from the $M_2$-polarization restrict to $7$ classes on $Z$ that have not already been accounted for.
\end{enumerate}
Combining all of the above together, we conclude that the image of the pull-back $i^*\colon H^{1,1}(Y) \to H^{1,1}(Z)$ spans a subspace of $H^{1,1}(Z)$ with dimension $1+6+14+7 = 28$.

Using this, together with a diagram chase in the localization exact sequence (Proposition \ref{prop:perverselocalization}), we can compute the mixed Hodge structures $(W_{\bullet},F^{\bullet})$ on the graded pieces of the perverse filtration $P_{\bullet}$.  For the cohomology of $Y$ we have the following.

\begin{center}
\begin{TAB}{|c|c|c|c|c|c|c|}{|c|t|}
$\Gr_2^PH^0(Y)$&&$\Gr_2^PH^2(Y)$&$\Gr_3^PH^2(Y)$&$\Gr_4^PH^2(Y)$ && $\Gr_3^PH^3(Y)$\\
\begin{tabular}{c|c}
\diagbox[innerleftsep=1pt,innerrightsep=1pt,width=18pt,height=18pt]{\scriptsize{$F$}}{\scriptsize{$W$}} & $0$\\
\hline
$0$& $1$
\end{tabular}&&
\begin{tabular}{c|c}
\diagbox[innerleftsep=1pt,innerrightsep=1pt,width=18pt,height=18pt]{\scriptsize{$F$}}{\scriptsize{$W$}} & $2$\\
\hline
$1$& $74$
\end{tabular}&
\begin{tabular}{c|c}
\diagbox[innerleftsep=1pt,innerrightsep=1pt,width=18pt,height=18pt]{\scriptsize{$F$}}{\scriptsize{$W$}} & $2$\\
\hline
$1$& $26$ 
\end{tabular}&
\begin{tabular}{c|c}
\diagbox[innerleftsep=1pt,innerrightsep=1pt,width=18pt,height=18pt]{\scriptsize{$F$}}{\scriptsize{$W$}} & $2$\\
\hline
$1$& $1$
\end{tabular}
&&
\begin{tabular}{c|c}
\diagbox[innerleftsep=1pt,innerrightsep=1pt,width=18pt,height=18pt]{\scriptsize{$F$}}{\scriptsize{$W$}} & $3$\\
\hline
$0$& $1$\\
$1$ & $1$\\
$2$&$1$\\
$3$ & $1$
\end{tabular}
\end{TAB}
\smallskip

\begin{TAB}{|c|c|c|c|c|}{|c|t|}
$\Gr_2^PH^4(Y)$&$\Gr_3^PH^4(Y)$&$\Gr_4^PH^4(Y)$ && $\Gr_4^PH^6(Y)$\\
\begin{tabular}{c|c}
\diagbox[innerleftsep=1pt,innerrightsep=1pt,width=18pt,height=18pt]{\scriptsize{$F$}}{\scriptsize{$W$}} & $4$\\
\hline
$2$& $1$
\end{tabular}&
\begin{tabular}{c|c}
\diagbox[innerleftsep=1pt,innerrightsep=1pt,width=18pt,height=18pt]{\scriptsize{$F$}}{\scriptsize{$W$}} & $4$\\
\hline
$2$ & $26$
\end{tabular}&
\begin{tabular}{c|c}
\diagbox[innerleftsep=1pt,innerrightsep=1pt,width=18pt,height=18pt]{\scriptsize{$F$}}{\scriptsize{$W$}} & $4$\\
\hline
$2$& $74$
\end{tabular}
&&
\begin{tabular}{c|c}
\diagbox[innerleftsep=1pt,innerrightsep=1pt,width=18pt,height=18pt]{\scriptsize{$F$}}{\scriptsize{$W$}} & $6$\\
\hline
$3$& $1$
\end{tabular}
\end{TAB}
\end{center}

Finally, for the compactly supported cohomology of $U$ we have the following filtrations. The filtrations on the cohomology of $U$ can be obtained from these by Poincar\'{e}-Verdier duality.

\begin{center}
\begin{TAB}{|c|c|c|}{|c|t|}
$\Gr_2^PH^2_c(U)$&&$\Gr_3^PH^3_c(U)$\\
\begin{tabular}{c|c}
\diagbox[innerleftsep=1pt,innerrightsep=1pt,width=18pt,height=18pt]{\scriptsize{$F$}}{\scriptsize{$W$}} & $2$\\
\hline
$1$& $73$
\end{tabular}&&
\begin{tabular}{c|cc}
\diagbox[innerleftsep=1pt,innerrightsep=1pt,width=18pt,height=18pt]{\scriptsize{$F$}}{\scriptsize{$W$}} & $2$&$3$\\
\hline
$0$& $2$ & $1$\\
$1$ & $2$ & $1$\\
$2$ & $2$ & $1$\\
$3$ & $0$ & $1$
\end{tabular}
\end{TAB}
\smallskip

\begin{TAB}{|c|c|c|c|c|}{|c|t|}
$\Gr_2^PH^4_c(U)$&$\Gr_3^PH^4_c(U)$&$\Gr_4^PH^4_c(U)$&& $\Gr_4^PH^6_c(U)$\\
\begin{tabular}{c|c}
\diagbox[innerleftsep=1pt,innerrightsep=1pt,width=18pt,height=18pt]{\scriptsize{$F$}}{\scriptsize{$W$}} & $4$\\
\hline
$2$& $1$
\end{tabular}&
\begin{tabular}{c|c}
\diagbox[innerleftsep=1pt,innerrightsep=1pt,width=18pt,height=18pt]{\scriptsize{$F$}}{\scriptsize{$W$}} & $4$\\
\hline
$2$& $26$
\end{tabular}&
\begin{tabular}{c|c}
\diagbox[innerleftsep=1pt,innerrightsep=1pt,width=18pt,height=18pt]{\scriptsize{$F$}}{\scriptsize{$W$}} & $4$\\
\hline
$2$& $73$
\end{tabular}&&
\begin{tabular}{c|c}
\diagbox[innerleftsep=1pt,innerrightsep=1pt,width=18pt,height=18pt]{\scriptsize{$F$}}{\scriptsize{$W$}} & $6$\\
\hline
$3$& $1$
\end{tabular}
\end{TAB}
\end{center}

Comparing to the computation in Subsection \ref{sec:type3CY3}, one can verify that Conjecture \ref{con:mirrorfiltrations} holds in this setting. This completes the proof of Theorem \ref{thm:CY3evidence}(2).

\bibliographystyle{amsalpha}
\bibliography{publications,preprints}
\end{document}